\documentclass[12pt, a4]{amsart}

\usepackage{latexsym,amsmath,amssymb,amsthm,mathrsfs,color}
\usepackage{float}
\usepackage{graphicx}
\usepackage{marginnote}

\usepackage[bbgreekl]{mathbbol}
\usepackage{bbm}
\usepackage{tikz-cd}
\usepackage{comment}
\usetikzlibrary{graphs,decorations.pathmorphing,decorations.markings}

\usepackage[all]{xy}
\usepackage{stmaryrd}

\usepackage{hyperref}
\hypersetup{colorlinks,%
    citecolor=brown,%
    filecolor=black,%
   linkcolor=blue,%
   urlcolor=black}

\usepackage[toc,page]{appendix}

\usepackage[shortalphabetic]{amsrefs}
 \makeatletter
\def\append@label@year@{%
    \safe@set\@tempcnta\bib@year
    \edef\bib@citeyear{\the\@tempcnta}%
    \ifnum\bib@citeyear>9
      \append@to@stem{%
          \ifx\bib@year\@empty
          \else
            \@xp\year@short \bib@citeyear \@nil
          \fi
      }%
    \fi
}
\makeatother

\let\oldtocsection=\tocsection
\renewcommand{\tocsection}[2]{\hspace{0em}\oldtocsection{#1}{#2}}

\usepackage[a4paper]{geometry}

\def\upddots{\mathinner{\mkern 1mu\raise 1pt \hbox{.}\mkern 2mu
\mkern 2mu \raise 4pt\hbox{.}\mkern 1mu \raise 7pt\vbox {\kern 7
pt\hbox{.}}} }

\numberwithin{equation}{section}

\begin{document}
\setlength{\unitlength}{2.5cm}

\newtheorem{thm}{Theorem}[section]
\newtheorem{lm}[thm]{Lemma}
\newtheorem{prop}[thm]{Proposition}
\newtheorem{cor}[thm]{Corollary}
\newtheorem{conj}[thm]{Conjecture}
\newtheorem{specu}[thm]{Speculation}

\theoremstyle{definition}
\newtheorem{dfn}[thm]{Definition}
\newtheorem{eg}[thm]{Example}
\newtheorem{rmk}[thm]{Remark}

\newcommand{\F}{\mathbf{F}}
\newcommand{\N}{\mathbbm{N}}
\newcommand{\R}{\mathbbm{R}}
\newcommand{\C}{\mathbbm{C}}
\newcommand{\Z}{\mathbbm{Z}}
\newcommand{\Q}{\mathbbm{Q}}
\newcommand{\Mp}{{\rm Mp}}
\newcommand{\Sp}{{\rm Sp}}
\newcommand{\GSp}{{\rm GSp}}
\newcommand{\GL}{{\rm GL}}
\newcommand{\PGL}{{\rm PGL}}
\newcommand{\SL}{{\rm SL}}
\newcommand{\SO}{{\rm SO}}
\newcommand{\Spin}{{\rm Spin}}
\newcommand{\GSpin}{{\rm GSpin}}
\newcommand{\Ind}{{\rm Ind}}
\newcommand{\Res}{{\rm Res}}
\newcommand{\Hom}{{\rm Hom}}
\newcommand{\End}{{\rm End}}
\newcommand{\msc}[1]{\mathscr{#1}}
\newcommand{\mfr}[1]{\mathfrak{#1}}
\newcommand{\mca}[1]{\mathcal{#1}}
\newcommand{\mbf}[1]{{\bf #1}}
\newcommand{\mbm}[1]{\mathbbm{#1}}
\newcommand{\into}{\hookrightarrow}
\newcommand{\onto}{\twoheadrightarrow}
\newcommand{\s}{\mathbf{s}}
\newcommand{\cc}{\mathbf{c}}
\newcommand{\bfa}{\mathbf{a}}
\newcommand{\id}{{\rm id}}
\newcommand{\g}{ \mathbf{g} }
\newcommand{\w}{\mathbbm{w}}
\newcommand{\Ftn}{{\sf Ftn}}
\newcommand{\p}{\mathbf{p}}
\newcommand{\bq}{\mathbf{q}}
\newcommand{\WD}{\text{WD}}
\newcommand{\W}{\text{W}}
\newcommand{\Wh}{{\rm Wh}}
\newcommand{\Whc}{{{\rm Wh}_\psi}}
\newcommand{\ggma}{\omega}
\newcommand{\sct}{\text{\rm sc}}
\newcommand{\Of}{\mca{O}^\digamma}
\newcommand{\gk}{c_{\sf gk}}
\newcommand{\Irr}{ {\rm Irr} }
\newcommand{\Irrg}{ {\rm Irr}_{\rm gen} }
\newcommand{\diag}{{\rm diag}}
\newcommand{\uchi}{ \underline{\chi} }
\newcommand{\Tr}{ {\rm Tr} }
\newcommand{\der}\de
\newcommand{\Stab}{{\rm Stab}}
\newcommand{\Ker}{{\rm Ker}}
\newcommand{\bfp}{\mathbf{p}}
\newcommand{\bfq}{\mathbf{q}}
\newcommand{\KP}{{\rm KP}}
\newcommand{\Sav}{{\rm Sav}}
\newcommand{\de}{{\rm der}}
\newcommand{\tnu}{{\tilde{\nu}}}
\newcommand{\lest}{\leqslant}
\newcommand{\gest}{\geqslant}
\newcommand{\tu}{\widetilde}
\newcommand{\tchi}{\tilde{\chi}}
\newcommand{\tomega}{\tilde{\omega}}
\newcommand{\Rep}{{\rm Rep}}
\newcommand{\cu}[1]{\textsc{\underline{#1}}}
\newcommand{\set}[1]{\left\{#1\right\}}
\newcommand{\ul}[1]{\underline{#1}}
\newcommand{\wt}[1]{\overline{#1}}
\newcommand{\wtsf}[1]{\wt{\sf #1}}
\newcommand{\anga}[1]{{\left\langle #1 \right\rangle}}
\newcommand{\angb}[2]{{\left\langle #1, #2 \right\rangle}}
\newcommand{\wm}[1]{\wt{\mbf{#1}}}
\newcommand{\elt}[1]{\pmb{\big[} #1\pmb{\big]} }
\newcommand{\ceil}[1]{\left\lceil #1 \right\rceil}
\newcommand{\floor}[1]{\left\lfloor #1 \right\rfloor}
\newcommand{\val}[1]{\left| #1 \right|}
\newcommand{\aff}{ {\rm aff} }
\newcommand{\ex}{ {\rm ex} }
\newcommand{\exc}{ {\rm exc} }
\newcommand{\HH}{ \mca{H} }
\newcommand{\HKP}{ {\rm HKP} }
\newcommand{\std}{ {\rm std} }
\newcommand{\motimes}{\text{\raisebox{0.25ex}{\scalebox{0.8}{$\bigotimes$}}}}

\newcommand{\FF}{\mca{F}}
\newcommand{\tv}{\tilde{v}}

\title[Gelfand--Graev functor and quantum affine Schur--Weyl duality]{Genuine Gelfand--Graev functor and the quantum affine Schur--Weyl duality}

\author{Fan Gao, Nadya Gurevich, and Edmund Karasiewicz}
\address{F. Gao: School of Mathematical Sciences, Zhejiang University, 866 Yuhangtang Road, Hangzhou, China 310058}
\email{gaofan@zju.edu.cn}
\address{N. Gurevich: Department of Mathematics, Ben Gurion University of the Negev, Be'er Sheva,  Israel 8410501}
\email{ngur@math.bgu.ac.il }
\address{E. Karasiewicz: Department of Mathematics, University of Utah, Salt Lake City, USA 84112}
\email{karasiewicz@math.utah.edu}

\date{}
\maketitle

\begin{abstract} 
We explicate relations among the Gelfand--Graev modules for central covers, the Euler--Poincar\'e polynomial of the Arnold--Brieskorn manifold, and the quantum affine Schur--Weyl duality. These three objects and their relations are dictated by a permutation representation of the Weyl group. 

Specifically, our main result shows that for certain covers of $\mathrm{GL}(r)$ the Gelfand--Graev functor is related to quantum affine Schur--Weyl duality. Consequently, the commuting algebra of the Iwahori-fixed part of the Gelfand--Graev representation is the quotient of a quantum group.
\end{abstract}

\tableofcontents


\section{Introduction}
It might be appropriate to retitle the paper as ``A short tale of a permutation representation of the Weyl group", since the main content concerns objects and their relations as depicted in the following diagram:

\begin{equation} \label{PIC}
\begin{tikzcd}
&   &({\rm ind}_{\mu_n \times U^-}^{\wt{G}} \iota\times \psi)^I \ar[dd, dashed, "{f_{23}}"] \\
{\rm EP}(\msc{M}_{\rm AB}, X) \ar[rru, dashed, bend left=13, "{f_{12}}"] \ar[rrd, dashed, bend right=13, "{f_{13}}"'] & W \curvearrowright Y/Y_{Q,n}  \ar[ur, dashed, "{f_{02}}"]  \ar[dr, dashed, "{f_{03}}"] \ar[l, dashed, "{f_{01}}"] & \\
&   & \mca{F}_{\rm SW}^\Psi,
\end{tikzcd}
\end{equation}
where
\begin{enumerate}
\item[(i)] $W \curvearrowright Y/Y_{Q,n}$ denotes the permutation representation
$$\sigma_\msc{X}: W \longrightarrow  {\rm Perm}(\msc{X}_{Q,n})$$
 of the Weyl group $W$ of a root system $\Phi$ acting on $\msc{X}:=\msc{X}_{Q,n}:=Y/Y_{Q,n}$ with $Y$ the cocharacter lattice and $Y_{Q,n} \subset Y$ a certain sublattice (see \S \ref{S:FWh} for details);
\item[(ii)] ${\rm EP}(\msc{M}_{\rm AB}, X)$ denotes the Euler--Poincar\'e polynomial of the Arnold--Brieskorn manifold $\msc{M}_{\rm AB}$ (see \cite{Arn69, Bri73}), the complement of the full Coxeter hyperplane arrangement of the root system $\Phi$;
\item[(iii)] $({\rm ind}_{\mu_n \times U^-}^{\wt{G}} \iota\times \psi)^I=:\mca{V}$ is the Iwahori-component of the Gelfand--Graev representation of an $n$-fold central cover $\wt{G}$ of a connected reductive group $G$ with root system $\Phi$, see \cite{GGK};
\item[(iv)] $\mca{F}_{\rm SW}^\Psi: \mca{M}_L(\HH(S_r^\aff))  \to \mca{M}_L(U_q(\hat{\mfr{sl}}(n_\alpha))$ is the quantum affine Schur--Weyl duality functor of left modules, defined for the Iwahori--Hecke algebra $\HH(S_r^\aff)$ of $\GL_r$ and the quantum affine group $U_q(\hat{\mfr{sl}}(n_\alpha))$, see \cite{ChPr96, DDF12, Ant}.
\end{enumerate}

Now we elaborate on our main results relating the above objects. Note, some notation is taken from \cite{GGK}. We provide precise references when needed.


Henceforth, we fix a $p$-adic field $F$ with prime ideal $\mfr{p}_F$ and ring of integers $O_F$. We have $G:=\mbf{G}(F)$ for a smooth connected split reductive group scheme $\mbf{G}$ over $O_F$ with root system $\Phi$. Assume that $F^\times $ contains the full group $\mu_n$ of $n$-th roots of unity and $\gcd(p, n)=1$. For a split maximal torus $\mbf{T}\subset \mbf{G}$ let 
$$Y:=\Hom(\mbf{G}_m, \mbf{T})$$
denote the cocharacter lattice. Let $W$ be the Weyl group of $(\mbf{G},\mbf{T})$. Partially depending on a $W$-invariant quadratic form 
$$Q: Y \longrightarrow \Z,$$
one has a Brylinski--Deligne central extension
$$\begin{tikzcd}
\mu_n \ar[r, hook] & \wt{G}^{(n)} \ar[r, two heads,"\mathrm{pr}"] & G.
\end{tikzcd}$$
For any $H \subset G$, let $\wt{H}=\mathrm{pr}^{-1}(H)$; in particular, $\wt{G}=\wt{G}^{(n)}$. For every root $\alpha$, we define
$$n_\alpha:=\frac{n}{\gcd(n, Q(\alpha^\vee))} \in \N,$$
which appears in (iv) and only depends on the length of $\alpha$.  

Throughout, we fix a faithful character $\iota: \mu_n \into \C^\times$ and consider $\iota^{\pm1}$-genuine representations of $\wt{G}$, meaning $\mu_n$ acts via $\iota^{\pm1}$. We write $\mathrm{Irr}_{\iota}(\overline{G})$ for the set of isomorphism classes of irreducible $\iota$-genuine smooth representations of $\overline{G}$.

Let $U^- \subset B^-=TU^-$ be the opposite unipotent radical of a fixed Borel subgroup $B=TU$. We fix a non-degenerate character
$$\psi: U^- \longrightarrow \C$$
of conductor $\mfr{p}_F$, i.e., $\psi: U_{-\alpha} \to \C$ is of conductor $\mfr{p}_F$ for every simple root $\alpha$. The cover $\wt{G}$ splits canonically over $U$ and $U^-$, which are thus viewed as subgroups of $\wt{G}$. For every $\sigma \in \Irr_{\iota^{-1}}(\wt{G})$, we define the $\psi$-Whittaker space of $\sigma$ to be
$$\Wh_\psi(\sigma):=\Hom_{\wt{G}}({\rm ind}_{\mu_n \times U^-}^{\wt{G}} \iota \times \psi^{-1}, \sigma^\vee),$$
where $\sigma^\vee \in \Irr_{\iota}(\wt{G})$ is the contragredient representation of $\sigma$.

Define $K:=\mbf{G}(O_F)$, and let $I$ be the preimage of $\mbf{B}(O_{F}/\frak{p}_{F})$ under the mod $\frak{p}_{F}$ reduction map $K\rightarrow \mbf{G}(O_{F}/\frak{p}_{F})$. We fix a splitting $s_K:K \into \wt{G}$ and identify $K$ and $I$ with their image in $\wt{G}$. Let $C^{\infty}_{c}(\overline{G})$ be the space of smooth compactly supported functions on $\overline{G}$. Then
$$\HH_I:= C_{c, \iota}^\infty(I\backslash \wt{G} /I)=\{f\in C^{\infty}_{c}(\overline{G})|f(\zeta\gamma_{1}g\gamma_{2})=\iota(\zeta)f(g)\text{ for all }g\in \overline{G},\gamma_{j}\in I,\zeta\in \mu_{n}\},$$
is the $\iota$-genuine Iwahori--Hecke algebra. By Borel and Casselman, the map 
$$\sigma \mapsto \sigma^I$$
gives a bijection between $\iota^{-1}$-genuine irreducible Iwahori-spherical representations of $\wt{G}$ and simple modules over $\HH_I$, see  \cite[Proposition (17.1)]{FlKa86} or \cite[Corollary 6.8]{Sav04}. 

We also consider the Hecke algebra $\widehat{\HH}_{I}$ obtained from $\HH_{I}$ by precomposing by the antiinvolution $\widehat{\phantom{X}}:\overline{G}\rightarrow\overline{G}$ defined by $g\mapsto g^{-1}$. Note that $\widehat{\HH}_{I}$ consists of $\iota^{-1}$-genuine functions. The introduction of $\widehat{\HH}_{I}$ is not essential; it allows us to more easily align our results with the existing literature on quantum groups. Every right $\HH_I$-module is canonically viewed as a left $\widehat{\HH}_I$-module.

In \cite{GGK}, we determined the structure of 
$\mca{V}=({\rm ind}_{\mu_n \times U^-}^{\wt{G}} \iota\times \psi)^I$ as a right $\HH_I$-module (equivalently, as a left $\widehat{\HH}_{I}$-module) for $\msc{X}$-splitting covers (see Definition \ref{D:Xspl}). In particular, the permutation representation $\sigma_\msc{X}$ in (i) underlies this structure. We explain this in \S \ref{S:FWh}, which gives the relation $f_{02}$ in \eqref{PIC}. Note, in this paper we use $\mca{V}$, while we used $\mca{V}^I$ in \cite{GGK} for the same module. 

For any algebra $A$, we denote by $\mca{M}_L(A), \mca{M}_R(A)$ the category of left and right modules over $A$ respectively. 

We show in Proposition \ref{P:Meq} that if $n$ is stable (see Definition \ref{D:stab}), then the functor $\mathrm{Hom}_{\HH_{I}}(\mathcal{V},-)$ gives an equivalence of categories between $\mathcal{M}_{R}(\HH_{I})$ and $\mathcal{M}_{R}(\mathrm{End}_{\HH_{I}}(\mca{V}))$. So, in particular $\HH_I$ and $\End_{\HH_I}(\mca{V})$ are Morita equivalent. For Kazhdan--Patterson covers and Savin covers of $\GL_r$, stability is equivalent to the inequality $n_\alpha\gest  r$, see Example \ref{E:GLsta}. Such results provide a partial understanding of the affine $q$-Schur algebra $\End_{\mca{H}_I}(\mca{V})$. In the linear case, Bushnell and Henniart determine the structure of endomorphism algebra of the Gelfand--Graev representation  for general Bernstein classes \cite[Theorem 4.3]{BuHe03}.

In \S \ref{S:EP}, we determine the stability of $n$ for covers $\wt{G}^{(n)}$ of semisimple simply-connected $G$. To this end, we must compute the Whittaker dimension
$$\dim \Wh_\psi(\Theta(\wt{G}^{(n)}))$$
of the theta representation. Using Sommers' computation of the character of the permutation representation $\sigma_\msc{X}$ \cite{Som97}, we show the following.
\begin{thm}[Theorem \ref{T:Whd}] \label{T:M1}
For $\wt{G}^{(n)}$ an oasitic cover of a semisimple simply-connected $G$ with $Q(\alpha^\vee)=-1$ for any long root $\alpha$, one has
$$\dim \Wh_\psi(\Theta(\wt{G}^{(n)})) = \val{W}^{-1} \cdot n^r \cdot {\rm EP}(\msc{M}_{\rm AB}, -n^{-1}) = \val{W}^{-1} \cdot \prod_{j=1}^r (n - m_j),$$
where the $m_j$'s are the exponents of the Weyl group $W$. 
\end{thm}
The above dimension is related to the characteristic polynomial of the Coxeter hyperplane arrangment by results of Orlik--Solomon \cite{OrSo83}.  In Theorem \ref{T:Whd}, we also give an analogous formula for the Steinberg representation ${\rm St}(\wt{G}^{(n)})$ in terms of ${\rm EP}(\msc{M}_{\rm AB}, X)$. This yields the link $f_{01}$.  Moreover, by an alternative characterization of stability, Theorem \ref{T:M1} enables us to determine when $\HH_I \subseteq \mca{V}$ as $\HH_I$-modules. This provides the connection $f_{12}$ in \eqref{PIC}.

Now we turn to the main result in our paper. Arising from the Gelfand--Graev module $\mca{V}$ is the functor
$$ \FF_{\rm GG}:   \mca{M}_L(\widehat{\HH}_I) \longrightarrow \mca{M}_L( \End_{\widehat{\mca{H}}_I}(\mca{V}^{*}))$$
given by 
$$\pi \mapsto \mathrm{Hom}_{\widehat{\HH}_{I}}(\mca{V},\pi)\simeq \mca{V}^{*}\otimes_{\widehat{\HH}_{I}}\pi,$$
where $\mca{V}^{*}=\mathrm{Hom}_{\widehat{\HH}_{I}}(\mca{V},\widehat{\HH}_{I})$. We will call this the Gelfand--Graev functor. This functor is closely related to the Whittaker functor
\begin{equation*}
\FF_{\rm Wh}:   \mca{M}_R(\widehat{\HH}_I) \longrightarrow \mca{M}_L( \End_{\widehat{\mca{H}}_I}(\mca{V}^{*}))
\end{equation*}
given by 
\begin{equation*}
\pi\mapsto \mathrm{Hom}_{\widehat{\HH}_{I}}(\mca{V},\pi^{\vee})\simeq \mca{V}^{*}\otimes_{\widehat{\HH}_{I}}\pi^{\vee},
\end{equation*}
where $\pi^\vee$ is canonically viewed as left $\widehat{\HH}_I$-module with action specified by $\angb{v}{h\cdot \lambda} = \angb{v\cdot h}{\lambda}$ for $v\in \pi, \lambda \in \pi^\vee$ and $h \in \widehat{\HH}_I$.

In \S \ref{S:SW}, we recall the quantum affine Schur--Weyl functor $\mca{F}_{\rm SW}$ for $\GL_r$ which arises from the vector space $\msc{V}_{\rm SW}\simeq (\C[\Z])^{\otimes r}$ endowed with commuting actions
\begin{equation} \label{c-act}
U_q(\hat{\mfr{sl}}(m)) \curvearrowright \msc{V}_{\rm SW}^{[m]} \curvearrowleft \HH(S_r^\aff)
\end{equation}
for general $m, r \in \N$. We highlight that both actions of $\HH(S_r^\aff)$ and $U_q(\hat{\mfr{sl}}(m))$ on $\msc{V}_{\rm SW}$ depends on $m$. Thus, we use $\msc{V}_{\rm SW}^{[m]}$ in \eqref{c-act} and in the paper to emphasize this dependence. Also, the presence of $\sigma_\msc{X}$ in describing the action of $\HH(S_r^\aff)$ on $\msc{V}_{\rm SW}^{[m]}$ is clear by a simple translation of terminology from that in \cite{Gre99, DDF12, Ant} for example, see the proof of Theorem \ref{T:SW=Wh}. This gives the link $f_{03}$.

In any case, one has a natural functor
$$\mca{F}_{\rm SW}: \mca{M}_L(\HH(S_r^\aff))  \to \mca{M}_L(\End_{\HH(S_{r}^{\aff})}(\mathscr{V}_{\rm SW}^{[m]}))$$
given by $\sigma\mapsto \mathscr{V}_{\rm SW}^{[m]}\otimes_{\HH(S_{r}^{\aff})} \sigma$. As a consequence of \eqref{c-act}, 
we have the natural algebra homomorphism
  $$\Psi: U_q(\hat{\mfr{sl}}(m)) \longrightarrow \End_{\HH(S_r^\aff)}(\msc{V}_{\rm SW}^{[m]})$$
which induces a functor
  $$\Psi^*: \mca{M}_L(\End_{\HH(S_{r}^{\aff})}(\mathscr{V}_{\rm SW}^{[m]})) \longrightarrow \mca{M}_L(U_q(\hat{\mfr{sl}}(m))).$$
We set $\mca{F}_{\rm SW}^\Psi:= \Psi^* \circ \mca{F}_{\rm SW}$, which clearly arises from \eqref{c-act}. Note that $\mca{F}_{\rm SW}, \Psi$ and $\mca{F}_{\rm SW}^\Psi$ all depend on $m$.

Our main result (Theorem \ref{T:SW=Wh}) is that for certain covers of type A groups the two functors $\mca{F}_{\rm GG}, \mca{F}_{\rm SW}$ are closely related, where we take $m=n_\alpha$ for the latter functor. More precisely,  for covers of $\wt{\GL}_r^{(n)}$ of type (\hyperref[C1]{C1}) (see \S \ref{S:FWh}), one has a natural algebra isomorphism
$$\phi_\theta: \mca{H}(S_r^\aff) \rightarrow \widehat{\HH}_{I},$$
which enables us to compare $\mca{V}^{*}$ and $\msc{V}_{\rm SW}^{[n_\alpha]}$. For this, we show in Proposition \ref{P:IM} that
\begin{equation} \label{IM-2V}
 \mca{V}^{*}  \simeq (\phi_\theta^{-1})^*\circ{\rm IM}^*(\msc{V}_{\rm SW, \theta}^{[n_\alpha]})
 \end{equation}
as right $\widehat{\HH}_I$-modules, where IM stands for the Iwahori--Mastumoto involution on $\HH(S_r^\aff)$. 
This enables us to give a natural identification $\varphi_\theta: \End_{\widehat{\HH}_I}(\mca{V}^{*}) \to \End_{H(S_r^\aff)}(\msc{V}_{\rm SW}^{[n_\alpha]})$. As mentioned, the main result in \S \ref{S:SW} is

\begin{thm}[Theorem \ref{T:SW=Wh}]\label{T:M2}
Let $\wt{G}=\wt{\GL}_r^{(n)}$ be a cover of $\GL_r$ of type (\hyperref[C1]{C1}).  Then for every $\pi\in \mca{M}_L(\widehat{\HH}_I)$, one has
\begin{equation} \label{E:SW=Wh}
\mca{F}_{\rm SW} \circ {\rm IM}^* \circ \phi_\theta^*(\pi) \simeq \varphi_\theta^* \circ \mca{F}_{\rm GG}(\pi),
\end{equation}
where $\mca{F}_{\rm SW}$ is associated with $m=n_\alpha$.
As a diagram of commuting functors we have:
$$\begin{tikzcd}
\mca{M}_L(\widehat{\HH}_I) \ar[d, "{\rm{IM}^{*}\circ\phi_\theta^*}"]  \ar[r, "\mathcal{F}_{\rm{GG}}"] & \mca{M}_L(\End_{\widehat{\mca{H}}_I}(\mca{V})) \ar[d, "{\varphi_\theta^*}"] \\
\mca{M}_L(\HH(S_r^\aff)) \ar[r, "{\mca{F}_{\rm SW}}"] &  \mca{M}_L(\End_{\HH(S_r^\aff)}(\msc{V}_{\rm SW}^{[n_\alpha]}))
\end{tikzcd}$$
\end{thm}

Since $\msc{V}_{\rm SW}^{[n_\alpha]}$ is a $(U_q(\hat{\mfr{sl}}(n_\alpha)),\HH(S_r^\aff))$-bimodule, $\mca{V}^{*}$ inherits a left $U_q(\hat{\mfr{sl}}(n_\alpha))$-action in view of \eqref{IM-2V}. This might be viewed as a process of quantization of the Whittaker functor: roughly speaking, one can further compose $\mca{F}_{\rm Wh}$ with $\Psi^*$. Following this, in \S \ref{SS:lsm-rm}  we utilize $\Psi, \mca{F}_{\rm SW}, \mca{F}_{\rm Wh}$ to identify the local scattering matrix associated with intertwining operators on genuine principal series $\tilde{\mbm{I}}_\theta(\chi_s), s=(s_1, ..., s_r) \in \C^r$ of $\wt{\GL}_r^{(n)}$ with an $R$-matrix, up to a flipping map. See Theorem \ref{T:sca} and the subsequent discussion. 

The result in Theorem \ref{T:sca} was already proved by Brubaker--Buciumas--Bump \cite{BBB19} by showing that the $R$-matrix (suitably twisted) gives a solution to the Yang--Baxter equation of metaplectic ice model, and it coincides with the local scattering matrix; thus, a natural map was defined from $\mca{F}_{\rm{GG}}(\tilde{\mbm{I}}_\theta(\chi_s)^I)$ to the quantum space 
$$V(n_\alpha s_1)\otimes ... \otimes V(n_\alpha s_r),$$
where $V(u)$ is the shifted standard evaluation $U_q(\hat{\mfr{sl}}(n_\alpha))$-module associated with any $u \in \C$, see Definition \ref{D:eva}. The map is equivariant with respect to intertwiners and homomorphisms associated with $R$-matrices. 

Here we give a reproof of \cite[Theorem 1]{BBB19}. Indeed, the functor $\mca{F}_{\rm SW}^\Psi$ essentially provides a natural substitute for the above map. Also, the functorial and monoidal properties of $\mca{F}_{\rm SW}^\Psi$, coupled with its relation with $\mca{F}_{\rm GG}$ (Theorem \ref{T:M2} above),  imply that the local scattering matrix represents a $U_q(\hat{\mfr{sl}}(n_\alpha))$-homomorphism. This gives Theorem \ref{T:sca}. The above results, especially the aforementioned ``quantization", give the connection $f_{23}$ in \eqref{PIC}.

In view of the relation between $\mca{F}_{\rm GG}, \mca{F}_{\rm Wh}$ and $\mca{F}_{\rm Wh}$, the polynomial ${\rm EP}(\msc{M}_{\rm AB}, X)$ determines the faithfulness of $\mca{F}_{\rm SW}$, see Remark \ref{R:GL}. This gives the connection $f_{13}$, which is in effect the composite of $f_{12}$ and $f_{23}$.

In the last section \S \ref{S:rmk}, we add several remarks regarding existing generalizations of $\mca{F}_{\rm SW}^\Psi$ to general Cartan types in the literature. Inevitably, we are not able to give exhaustive references on those deep topics, but can only content ourselves with a brief mention which pertains to our discussions above.

\subsection{Acknowledgement} We would like to thank Valentin Buciumas and Dennis Gaitsgory for some communications on an earlier version of the paper. We are especially grateful to Buciumas for very helpful comments and clarifications on the earlier work relevant to the topics discussed here. The work of F.\ G. is partially supported by the National Key R\&D Program of China (No. 2022YFA1005300) and also by NNSFC--12171422.

\section{Stability and the Gelfand--Graev module} \label{S:FWh}

In this section, let $\wt{G}=\wt{G}^{(n)}$ be a Brylinski--Deligne cover of $G$. Recall the bilinear form $B_Q(y, z):=Q(y+z) - Q(y) - Q(z)$ derived from $Q$ and the Weyl group stable sublattice
$$Y_{Q,n}=\set{y\in Y: B_Q(y, z) \in n\Z \text{ for all } z \in Y}.$$
The Weyl group $W$ acts naturally on $\msc{X}_{Q,n}=Y/Y_{Q,n}$. The $\iota$-genuine Iwahori--Hecke algebra $\HH_I$ can be viewed as the extended affine Hecke algebra associated with $Y_{Q,n} \rtimes W$. It contains the sub-algebra $\HH_W:=C_{c, \iota}^\infty(I \backslash \wt{K} /I)$, which is associated with $W$.

\begin{dfn} \label{D:Xspl}
An orbit $\mca{O}_y \subset \msc{X}_{Q,n}$ is called $\msc{X}$-splitting if the $W$-equivariant quotient map $f: Y \onto \msc{X}_{Q,n}$ has a $W$-equivariant section over $\mca{O}_y$.  A covering group $\wt{G}^{(n)}$ is called $\msc{X}$-splitting if every orbit in $\msc{X}_{Q,n}$ is $\msc{X}$-splitting, i.e., if the quotient map $f: Y \onto \msc{X}_{Q,n}$ has a $W$-equivariant section. 
\end{dfn}

Consider $\GL_r$ with lattice $Y$ given with the standard $\Z$-basis $\set{e_i: 1\lest i \lest r}$. A Brylinski--Deligne cover $\wt{\GL}_r^{(n)}$ of $\GL_r$ is associated with $\mbf{p}, \mbf{q}\in \Z$ such that
$$
B(e_i, e_j) =
\begin{cases}
2\mbf{p} & \text{ if } i =j,\\
\mbf{q} & \text{ if } i \ne j.
\end{cases}
$$
We have $Q(\alpha^\vee)=2\mbf{p} - \mbf{q}$. The Kazhdan--Patterson covers are those with $Q(\alpha^\vee)=-1$, where $\mbf{p}$ corresponds to the parameter $c$ in \cite{KP}. Meanwhile, Savin's ``nice" cover is the one with $\mbf{p}=-1, \mbf{q}=0$.
These covers are $\msc{X}$-splitting by \cite[Example 6.5]{GGK}. 

The following two classes of Brylinski--Deligne covers, which are also $\msc{X}$-splitting covers, are our main interest in this paper.
\begin{enumerate}
\item[(C1)]\label{C1} A Brylinski--Deligne cover $\wt{\GL}_r^{(n)}$ satisfying 
$$n| \mbf{q}$$
is called of type (C1). Such covers satisfy the block commutativity, i.e., for partition $\mbf{r}=(r_1, ..., r_k)$ of $r$ with associated standard Levi subgroup $M_\mbf{r}=\GL_{r_1} \times ... \times \GL_{r_k}$, one has
an isomorphism of groups
$$\Big(\prod_i \wt{\GL}_{r_i}^{(n)}\Big)/H \to \wt{M}_{\mbf{r}}^{(n)},$$
where $H=\set{(\zeta_i)_i: \prod_i \zeta_i =1}$. Moreover, for such covers
$$Y_{Q,n} = n_\alpha Y \text{ with } n_\alpha = n/\gcd(2\mbf{p}, n).$$
These covers are $\msc{X}$-splitting by \cite[Example 6.5]{GGK}. 

As concrete examples, recall that the Savin covers are Brylinski--Deligne covers with $(\mbf{p}, \mbf{q})=(-1, 0)$, and thus are of type (C1). On the other hand, for Kazhdan--Patterson covers to be of type (C1), $n$ is necessarily odd.
\item[(C2)]\label{C2} Let $G$ be semisimple and simply-connected. In \cite[Corollary 6.4]{GGK} we show that all oasitic covers (see Definition 6.1 of loc. cit.) of $G$ are $\msc{X}$-splitting. And in the tables of \cite[\S 6.1]{GGK}, we explicitly list the $n$'s such that an $n$-fold cover of such $G$ is oasitic.
\end{enumerate}

Now assume $\wt{G}$ is an $\msc{X}$-splitting cover. For every $W$-orbit $\mca{O}_y\subset \msc{X}_{Q,n}$, one may view $\mca{O}_y \subset Y$ such that 
$$W_y:={\rm Stab}_W(y, Y) ={\rm Stab}_W(y, \msc{X})$$ is a parabolic Weyl subgroup of $W$. Let 
$$\HH_{W_y} \subset \HH_W \subset \HH_I$$
be the subalgebra associated with $W_y$. Let $\varepsilon_{W_y}, \mbm{1}_{W_y}$ be the sign and trivial character of $\HH_{W_y}$ respectively. In \cite[Theorem 5.18]{GGK} we show that
\begin{equation} \label{E:VI}
\mca{V} \simeq \bigoplus_{\mca{O}_y \subset \msc{X}_{Q,n}} \varepsilon_{W_y} \otimes_{\HH_{W_y}} \HH_I
\end{equation}
as a right $\HH_I$-module. 

Set $\set{L, R} = \set{\Diamond, \Diamond'}$. For every $\iota^\pm$-genuine smooth representation $\pi$ of $\wt{G}$ with $\Diamond$-action, we denote by $\pi^\vee$ its contragredient representation which is $\iota^\pm$-genuine but with $\Diamond'$-action. More precisely, if $\Diamond=L$ for example, then $\pi^\vee$ consists of smooth functionals of $\pi$ with respect to the right action of $\wt{G}$ specified by
$$\angb{v}{ \lambda \cdot \pi^\vee(g)} = \angb{\pi(g) v}{\lambda}$$
for every $v \in \pi$ and $\lambda \in \pi^\vee$. 
For every $\iota$-genuine left $\wt{G}$-representation $\pi$, the $\psi$-Whittaker space of $\pi$ is given by
$$\Wh_\psi(\pi):=\Hom_{\wt{G}}({\rm ind}_{\mu_n \times U^-}^{\wt{G}} \iota^{-1} \times \psi^{-1}, \pi^\vee).$$

\begin{dfn} \label{D:stab}
Let $\wt{G}$ be an $\msc{X}$-splitting $n$-fold cover. The degree $n$ and also $\wt{G}$ are called stable if the following equivalent conditions are satisfied:
\begin{enumerate}
\item[(i)] there is a free $W$-orbit in $\msc{X}_{Q,n}$;
\item[(ii)] $\HH_I$ is a direct summand of $\mca{V}$ as $\HH_I$-modules;
\item[(iii)] every unramified theta representation $\Theta(\wt{G}, \chi)$ is generic,  i.e., $\Wh_\psi(\Theta(\wt{G}, \chi))\ne 0$.
\end{enumerate}
\end{dfn}
Here a theta representation $\Theta(\wt{G}, \chi)$ is the Langlands quotient of the genuine principal series $I(\chi)$ associated with an ``exceptional"  genuine central character $\chi$ of $Z(\wt{T})$, see \cite{KP, Ga2}. By \eqref{E:VI}, we have $\mca{O}_y \subset \msc{X}_{Q,n}$ is a free $W$-orbit if and only if $\mca{V}_{\mca{O}_y} \simeq \HH_I$, and this immediately gives the equivalence between (i) and (ii). The equivalence of (i) and (iii) is discussed in the proof of \cite[Theorem 8.8]{GGK}. It is easy to see that if $n$ is big enough, then it is stable since Definition \ref{D:stab} (i) is satisfied.

\begin{prop} \label{P:Meq}
Let $\wt{G}^{(n)}$ be an $\msc{X}$-splitting cover. If $n$ is stable, then $\mca{V}$ is a progenerator of $\mca{M}_R(\HH_I)$. In this case, the functor $\mathrm{Hom}_{\HH_{I}}(\mathcal{V},-):\mathcal{M}_{R}(\mathcal{H}_{I})\rightarrow \mathcal{M}_{R}(\End_{\HH_{I}}(\mca{V}))$ gives an equivalence of categories, and thus $\HH_I$ and $\End_{\HH_I}(\mca{V})$ are Morita-equivalent.
\end{prop}
\begin{proof}
The $\HH_I$-module $\mca{V}$ is a progenerator if $\mca{V}$ is projective, finitely generated, and is a generator (see \cite[\S 1.5.1]{Roc09}). The module $\mca{V}$ is projective by \cite[Lemma 8.6]{GGK}. Since $\mca{V}$ is projective, $\mca{V}$ is a generator if and only if for any nonzero $\sigma \in \mca{M}_R(\HH_I)$, the space 
$$\Hom_{\HH_I}(\mca{V}, \sigma) \ne 0.$$
Since $n$ is stable, this follows from Definition \ref{D:stab} (ii). Clearly, $\mca{V}$ is finitely-generated in view of its structure given in \eqref{E:VI}. 

Since $\mca{V}$ is a progenerator, it follows from general category theory (see \cite[\S 1.5.1]{Roc09}) that $\mathrm{Hom}_{\HH_{I}}(\mathcal{V},-)$ gives an equivalence of categories. This completes the proof.
\end{proof}

\begin{eg} \label{E:GLsta}
Let $\wt{\GL}_{r}$ be any $\msc{X}$-splitting cover, for example of Kazhdan--Patterson or Savin type, or of type (\hyperref[C1]{C1}). It follows from Lemma 3.1 and Theorem 3.7 of \cite{GaTs} that $n$ is stable if and only if $n_\alpha \gest r$. In fact, for general $\wt{G}$, the results in \cite[\S 3.1--3.5, Theorem 3.7]{GaTs} already give a lower bound for stable $n$. 
\end{eg}

\begin{rmk}
We have that $\HH_I \simeq \HH(G^\sharp, I^\sharp)$, where $\HH(G^\sharp, I^\sharp)$ is the Iwahori--Hecke algebra of a linear algebra group $G^\sharp$, see \cite{Sav04} or the discussion in the paragraph of \eqref{phi-t} later.  If $\wt{G}^{(n)}$ is a stable $\msc{X}$-splitting cover, then the Morita equivalence between $\HH_I$ and $\End_{\HH_I}(\mca{V})$ also follows from the work of Vigneras \cite{Vig03}.  For every subset $J \subset \Delta$ of the simple roots $\Delta$, let $\mbf{v}^\varepsilon_J \in \HH(G^\sharp, I^\sharp)$ (resp. $\mbf{v}^\mbm{1}_J$) denote the idempotent corresponding to $\varepsilon_{\HH_{W(J)}}$ (resp. $\mbm{1}$). This gives
$$\msc{S}^\varepsilon:=\End_{\HH(G^\sharp, I^\sharp)}\Big( \bigoplus_{J \subset \Delta} \mbf{v}^\varepsilon_J \HH(G^\sharp, I^\sharp) \Big).$$
Its analogue $\msc{S}^\mbm{1}$ was studied in \cite{Vig03, MiSt19}. For stable $n$, the algebra $\End_{\HH_I}(\mca{V})$ and $\msc{S}^\varepsilon$ are Morita equivalent and thus might be both called affine $q$-Schur algebras, as is the case for $\msc{S}^\mbm{1}$ in \cite{Vig03, MiSt19}. However, we note that they are not identical since $\mca{V}_{\mca{O}_y}$, which corresponds to certain
$\mbf{v}^\varepsilon_{J_y} \HH(G^\sharp, I^\sharp)$, has multiplicities  in $\mca{V}$ in general; that is, it is possible that $J_y = J_{y'}$ for $\mca{O}_y \ne \mca{O}_{y'}$. Such multiplicities are given in Sommers' work \cite{Som97} for oasitic covers of semisimple simply-connected $G$, which we will use in \S \ref{S:EP}.
In type A, such multiplicities are also important in relating the genuine Gelfand--Graev functor to the quantum affine Schur--Weyl functor, see \S \ref{S:SW} below.
\end{rmk}

Later in this paper it will be convenient at times to view $\mathcal{V}$ as a left module. This can be done as follows. Let 
$$\widehat{\phantom{X}}:C^{\infty}(\overline{G})\rightarrow C^{\infty}(\overline{G})$$
 be the anti-involution defined by $\hat{f}(g):=f(g^{-1})$. We write $\widehat{\HH}_I$ for the image of $\HH_{I}$ under this map. Note that $\widehat{\HH}_{I}$ is the $\iota^{-1}$-genuine Iwahori--Hecke algebra of $\wt{G}$ with respect to convolution. 
Every $\pi \in \mca{M}_R(\HH_I)$ is canonically viewed as a left $\widehat{\HH}_I$-module with action defined by 
$\hat{h}\cdot v:= v \cdot h$ for $v\in \pi$ and $\hat{h} \in \widehat{\HH}_I$.
If $A\subset \HH_{I}$ is a subalgebra, we write $\widehat{A}\subset \HH_{I}$ for its image in $\widehat{\HH}_{I}$. By a direct calculation we get the following description of $\mathcal{V}$ as a left $\widehat{\mathcal{H}}_{I}$-module.

\begin{prop}
As a left $\widehat{\mathcal{H}}_{I}$-module 
\begin{equation}
\mathcal{V}\simeq \bigoplus_{\mathcal{O}_{y}\subset \msc{X}_{Q,n}}\widehat{\mathcal{H}}_{I}\otimes_{\widehat{\mathcal{H}}_{W_{y}}} \varepsilon_{W_{y}}.
\end{equation}
\end{prop}

This perspective of left-modules will be advantageous when considering the Whittaker functor. We define the Whittaker functor
$$\begin{tikzcd}
\FF_{\rm Wh}:   \mca{M}_R(\widehat{\HH}_I) \ar[r] & \mca{M}_L( \End_{\widehat{\mca{H}}_I}(\mca{V})),
\end{tikzcd}$$
by
$$\FF_{\rm Wh}(\pi):= \Hom_{\widehat{\mca{H}}_I}(\mca{V}, \pi^{\vee}).$$
Here we view $\pi^\vee \in \mca{M}_L(\widehat{\HH}_L)$ with action specified by  $\angb{v}{h\cdot \lambda} = \angb{v \cdot h}{\lambda}$ for all $v\in \pi, \lambda \in \pi^\vee$ and $h\in \widehat{\HH}_I$.
The action of  $a \in \End_{\widehat{\HH}_I}(\mca{V})$ on $f\in \Hom_{\widehat{\HH}_I}(\mca{V}, \pi^\vee)$ is given by
$$(a\cdot f)= f \circ a.$$
Proposition \ref{P:Meq} also applies when $\mathcal{V}$ is viewed as a left $\widehat{\HH}_{I}$-module. In particular, $\mathcal{V}$ is a finitely generated projective left $\widehat{\HH}_{I}$-module. Thus by \cite[Page 271]{BouA1} we have
\begin{equation*}
\Hom_{\widehat{\HH}_I}(\mca{V}, \pi^\vee)\simeq \mca{V}^{*}\otimes_{\widehat{\HH}_{I}} \pi^{\vee},
\end{equation*}
where $\mca{V}^{*}=\mathrm{Hom}_{\widehat{\HH}_{I}}(\mca{V},\widehat{\HH}_{I})$ is the dual module with its natural right $\widehat{\HH}_{I}$-module structure.
Let 
$$\mca{F}_{\rm GG}:\mca{M}_{L}(\widehat{\HH}_{I})\rightarrow \mca{M}_{L}(\mathrm{End}_{\widehat{\HH}_{I}}(\mca{V}^*))$$ be the functor defined by 
\begin{equation*}
\mca{F}_{\rm GG}(\pi):=\mca{V}^{*}\otimes_{\widehat{\HH}_{I}} \pi.
\end{equation*}
We call this the Gelfand--Graev functor. It is clear $\mca{F}_{\mathrm{Wh}}(\pi)=\mca{F}_{\rm GG}(\pi^{\vee})$.

We end this section with a lemma collecting some identities to be used later in the paper.

\begin{lm} \label{L:V-mirr}
Let $\mathcal{O}_{y}\subset \msc{X}_{Q,n}$ be a splitting $W$-orbit. Let $\chi$ be a character of $\widehat{\mca{H}}_{W_{y}}$ and let $e_{\chi}\in \widehat{\mca{H}}_{W_{y}}$ be the idempotent associated with $\chi$.
\begin{enumerate}
\item[(i)] The map $\chi\otimes_{\widehat{\HH}_{W_{y}}}\widehat{\HH}_{I}\rightarrow e_{\chi}\widehat{\HH}_{I}$ defined by $1\otimes h\mapsto e_{\chi}h$ is an isomorphism of right $\widehat{\HH}_{I}$-modules.
\item[(ii)] The map $e_{\chi}\widehat{\HH}_{I}\rightarrow (\widehat{\HH}_{I}\otimes_{\widehat{\HH}_{W_{y}}} \chi)^{*}$ defined by $e_{\chi}h\mapsto (h^{\prime}\otimes 1\mapsto h^{\prime}e_{\chi}h)$ is an isomorphism of right $\widehat{\HH}_{I}$-modules.
\item[(iii)] Suppose that $\overline{G}$ is splitting. If we view $\mathcal{V}$ as a left $\widehat{\HH}_{I}$-module, then the dual module $\mathcal{V}^{*}$ is a right $\widehat{\HH}_{I}$-module and
\begin{equation*}
\mca{V}^{*}\simeq \bigoplus_{\mathcal{O}_{y}\subset \msc{X}_{Q,n}} \varepsilon_{y} \otimes_{\widehat{\HH}_{W_{y}}}\widehat{\HH}_{I}.
\end{equation*}
\end{enumerate}
\end{lm}


\section{An Euler--Poincar\'e polynomial and stability} \label{S:EP}

In this section, we consider exclusively a semisimple simply-connected $G$ and an oasitic cover $\wt{G}^{(n)}$ of $G$ (see \cite[Definition 6.1]{GGK}), such that for any short coroot $\alpha^\vee$,
$$Q(\alpha^\vee)=-1.$$
These are covers of type (\hyperref[C2]{C2}) in \S \ref{S:FWh}.

The unramified genuine exceptional characters $\chi$ of $Z(\wt{T})$ form a torsor over $Z(\wt{G}^\vee)$, which is trivial for oasitic covers; thus, there is a unique theta representation $\Theta(\wt{G}^{(n)})$, the unique irreducible quotient of $I(\chi)$ for the unique exceptional character $\chi$. The unique irreducible subrepresentation ${\rm St}(\wt{G}^{(n)})$ of $I(\chi)$ is the analogue of the Steinberg representation.

Using work of Sommers \cite{Som97}, we determine  $\dim \Wh_\psi(\Theta(\wt{G}^{(n)}))$ explicitly in terms of the Euler--Poincar\'e polynomial of the Coxeter hyperplane arrangement of $G$. This gives the stability condition for $n$ as well.

For semisimple simply-connected $G$, one has $Y=Y^{sc}$, which is $\Z$-spanned by the simple coroots. For each root $\alpha \in \Phi$, consider the hyperplane 
$$H_\alpha:=\set{v\in Y\otimes \R: \angb{v}{\alpha}=0}.$$
We have the set
$$\mca{A}:=\set{H_\alpha: \alpha\in \Phi}$$
of the full Coxeter hyperplane arrangements in $Y\otimes \R$. Let $\mca{L}=\mca{L}(\mca{A})$ be the set of intersections of hyperplanes in $\mca{A}$. We consider
$$Y\otimes \R \in \mca{L}$$
by taking the empty intersection of elements in $\mca{A}$. Elements in $\mca{L}$ are ordered, and we write $x < y$ if  $y \subseteq x$. We also use $x\lest y$ to represent $x<y$ or $x=y$. The Mobius function
$$\mu: \mca{L} \times \mca{L} \longrightarrow \Z$$
is uniquely defined by requiring $\mu(x, x)=1$,
$$\sum_{z:\ x\lest z\lest y}\mu(z, y)=0 \text{ if } x < y,$$
and $\mu(x, y)=0$ otherwise. The characteristic polynomial of $\mca{L}$ is thus given by
$$\omega(\mca{L}, X):=\sum_{x\in \mca{L}} \mu(Y\otimes \R, x) \cdot X^{\dim x}.$$

On the other hand, consider the Arnold--Brieskorn complex manifold
$$\msc{M}_{\rm AB}:= Y \otimes \C - \bigcup_{\alpha \in \Phi} H_\alpha\otimes \C$$
and its Euler--Poincar\'e polynomial
$${\rm EP}(\msc{M}_{\rm AB}, X):=\sum_{i\gest 0} \dim H^i(\msc{M}_{\rm AB}, \C) \cdot X^i.$$

\begin{thm} \label{T:Whd}
Let $\wt{G}^{(n)}$ be an oasitic cover of a semisimple simply-connected $G$ such that $Q(\alpha^\vee)=-1$ for any long root $\alpha$. Then one has
\begin{equation}  \label{E:Theta}
\dim \Wh_\psi(\Theta(\wt{G}^{(n)}))=  \frac{\omega(\mca{L}, n)}{\val{W}}= \frac{n^r\cdot {\rm EP}(\msc{M}_{\rm AB}, -n^{-1})}{\val{W}} = \val{W}^{-1} \cdot \prod_{j=1}^r (n-m_j),
\end{equation}
where $m_j, 1\lest j \lest r$ are the exponents of the Weyl group $W$. Similarly,
\begin{equation}  \label{E:St}
\dim \Wh_\psi({\rm St}(\wt{G}^{(n)}))=  \frac{(-1)^r \cdot \omega(\mca{L}, -n)}{\val{W}}= \frac{n^r\cdot {\rm EP}(\msc{M}_{\rm AB}, n^{-1})}{\val{W}} = \val{W}^{-1} \cdot \prod_{j=1}^r (n+m_j).
\end{equation}
\end{thm}
\begin{proof}
Before we begin, we note that we cite results from \cite{GGK} that include the assumption $Q(\alpha^{\vee})=1$, for short coroots. When $G$ is semisimple, these results also hold when $Q(\alpha^{\vee})=-1$, for short coroots.

First, we have
$$\omega(\mca{L}, X) = X^r \cdot {\rm EP}(\msc{M}_{\rm AB}, -X^{-1}) =\prod_{j=1}^r (X-m_j) \in \C[X],$$
where the first equality was a classical result of Orlik--Solomon \cite{OrSo83} and the second equality already follows from work of Brieskorn \cite[Th\'eor\`em 6]{Bri73}.  

It follows from \cite[Theorem 8.8]{GGK} that 
$$\dim \Wh_\psi(\Theta(\wt{G}^{(n)})) =\angb{\varepsilon_W}{ \sigma_\msc{X}}_W = \# \set{\text{free $W$-orbits in $\msc{X}_{Q,n}$}}.$$
Let $P_1, P_2, ..., P_k \subset W$ be the conjugacy classes of parabolic subgroups of $W$ such that
$$P_1 = W, \quad P_k=\set{1}.$$
Note that $\msc{X}_{Q,n} = Y^{sc}/nY^{sc}$. Since $\wt{G}^{(n)}$ is an oasitic cover (see \cite[\S 6.1]{GGK}), we see that $n$ is ``very good" in the terminology of \cite[Definition 3.5]{Som97}. Following this, we have from \cite[Lemma 4.2]{Som97} that 
$$\sigma_\msc{X}=\bigoplus_{i=1}^k \mfr{m}_i(n) \cdot \Ind_{P_i}^W(\mbm{1}_{P_i})$$
for well-defined $\mfr{m}_i(n) \in \Z_{\gest 0}$. Furthermore, it is shown in \cite[Proposition 5.1]{Som97}
that 
$$\mfr{m}_k(n) = \frac{\omega(\mca{L}, n)}{\val{W}}.$$
This gives
$$\angb{\varepsilon_W}{\sigma_\msc{X}}_W = \sum_{i=1}^k \mfr{m}_i(n) \cdot \angb{\varepsilon_{P_i}}{\mbm{1}_{P_i}} = \mfr{m}_k(n) = \frac{\omega(\mca{L}, n)}{\val{W}}.$$

Now we consider the Steinberg representation ${\rm St}(\wt{G}^{(n)})$. Again by \cite[Theorem 8.8]{GGK},
$$\dim \Wh_\psi({\rm St}(\wt{G}^{(n)})) =\angb{\mbm{1}_W}{ \sigma_\msc{X}}_W = \# \set{\text{all $W$-orbits in $\msc{X}_{Q,n}$}}.$$
For every $w\in W$, let $s(w)$ denote the least number of reflections whose product is $w$ and consider
$$d(w):=\dim (Y\otimes \R)^w,$$
the dimension of the set of fixed points of $w$ in $Y\otimes \R$. One has $d(w)= r -s(w)$. Let $\chi_{\sigma_\msc{X}}$ denote the character of $\sigma_\msc{X}$. It was shown in \cite[Proposition 3.9]{Som97} that
$$\chi_{\sigma_\msc{X}}(w)=n^{d(w)}$$
for every $w\in W$. We also note that
$$\varepsilon_W(w) = (-1)^{s(w)}.$$
Thus,
$$\begin{aligned}
\dim \Wh_\psi({\rm St}(\wt{G}^{(n)})) & =\frac{1}{\val{W}} \sum_{w\in W} n^{d(w)} \\
& = \frac{(-1)^r}{\val{W}} \sum_{w\in W} \varepsilon_W(w)\cdot (-n)^{d(w)} \\
& = \frac{(-1)^r}{\val{W}} \cdot \omega(\mca{L}, -n) \\
& = \val{W}^{-1} \cdot \prod_{j=1}^r (n + m_j).
\end{aligned}$$
This completes the proof.
\end{proof}

The above  recovers those explicit formulas for $\wt{\Sp}_{2r}$ and $\wt{G}_2$ computed in \cite[\S 8.2--8.3]{Ga2}. Applying $n=1$ in \eqref{E:St}, since ${\rm St}(G)$ is generic with unique Whittaker model, we get the classical Weyl order formula $\val{W}=\prod_{j=1}^r(1+m_j)$, see \cite[Corollary 2.3]{Sol66}.

Now assume that $m_{1}\lest m_{2}\lest \ldots\lest m_{r}$.
\begin{cor} \label{C:M-eq}
Let $\wt{G}^{(n)}$ be as in Theorem \ref{T:Whd}. Then $n$ is stable if and only if $n > m_r$, in which case 
$$\mathrm{Hom}_{\HH_{I}}(\mca{V},-):   \mca{M}_R(\HH_I) \longrightarrow \mca{M}_R( \End_{\mca{H}_I}(\mca{V}))$$ gives an equivalence of categories.
\end{cor}

\begin{proof}
If $n>m_{r}$, then $n$ is stable by Theorem \ref{T:Whd}. Conversely, the result follows by comparing \cite[Tables 1, 2]{GGK}, with the exponents of the associated Weyl group. We include the details for  type $E_{8}$. The other cases are similar.

By \cite[Table 2]{GGK}, if $\wt{G}$ is an oasitic cover of type $E_{8}$, then $n$ is not divisible by $2,3$, or $5$. The exponents of the $E_{8}$ Weyl group are $1,7,11,13,17,19,23,29$ (e.g. see \cite[Plate VII]{BouL2}). These are exactly the positive integers less than $30$ that are not divisible by $2,3$, or $5$. Since $n$ is stable, it follows that $n>29$, as desired.
\end{proof}

\begin{rmk}
Theorem \ref{T:Whd} is compatible with Example \ref{E:GLsta}, since $m_r=r$ for type $A_r$ groups. For $\wt{\GL}_r^{(n)}$, it is possible to compute $\chi_{\sigma_\msc{X}}$ similarly as in \cite{Som97} and thus obtain a formula for $\dim  \Wh_\psi(\Theta(\wt{\GL}_r, \chi))$ and $\dim \Wh_\psi({\rm St}(\wt{\GL}_r, \chi))$, where $\chi$ is an exceptional character. Alternatively, one can argue as follows in a special case. For this purpose, we consider a Kazhdan--Patterson cover $\wt{\GL}_r^{(n)}$ satisfying $\gcd(n, r) =1$. The covering subgroup $\wt{\SL}_r^{(n)}$
is thus an oasitic cover. In this case, the pair $(\wt{\GL}_r^{(n)}, \wt{\SL}_r^{(n)})$ is an isotypic pair in the sense of \cite[Definition 2.23]{GSS3}. In particular, one has 
\begin{equation} \label{E:iso-p}
\dim \Wh_\psi({\rm St}(\wt{\GL}_r^{(n)}, \chi)) = \val{Y/(Y^{sc} + Y_{Q,n})} \cdot \dim \Wh_\psi({\rm St}(\wt{\SL}_r^{(n)}, \chi_0))
\end{equation}
for the unique exceptional character $\chi_0$ for $\wt{\SL}_r^{(n)}$ obtained from the restriction of $\chi$. In fact, \eqref{E:iso-p} holds for any irreducible constituent (in particular, $\Theta(\wt{\GL}_r, \chi)$ here as well) of a regular unramified principal series of $\wt{\GL}_r$, see \cite[\S 4.3.1]{GSS3}. Combining \eqref{E:iso-p} and \eqref{E:St}, we get
$$\dim \Wh_\psi({\rm St}(\wt{\GL}_r, \chi)) = \frac{n}{\val{Y_{Q,n}/nY}} \cdot \frac{1}{ r!} \cdot \prod_{j=1}^{r-1}(n + j) = \frac{1}{\val{Y_{Q,n}/nY}} \cdot \binom{r+n-1}{r},$$
where $\val{Y_{Q,n}/nY} = \gcd(n, 2r\mbf{p} + r -1)$. We note that this formula can also be obtained from \cite[Theorem 4.7]{Zou22}, which gives the Whittaker dimensions of general square integrable representations of the Kazhdan--Patterson covers.
\end{rmk}

\section{Quantum affine Schur--Weyl duality and $\mca{F}_{\rm{GG}}$} \label{S:SW}
In this section, we consider exclusively a cover $\wt{G}=\wt{\GL}_r$ of type (\hyperref[C1]{C1}). The goal is to show that the Gelfand--Graev functor
$$\mca{F}_{\rm{GG}}: \mca{M}_L(\widehat{\HH}_I) \longrightarrow \mca{M}_L( \End_{\widehat{\mca{H}}_I}(\mca{V}))$$
defined by $\pi\mapsto \mca{V}^{*}\otimes_{\widehat{\HH}_{I}}\pi$ can be identified with the quantum affine Schur--Weyl functor 
 $$\mca{F}_{\rm SW}: \mca{M}_L(\HH(S_r^\aff))  \to \mca{M}_L(\End_{\HH(S_{r}^{\aff})}(\mathscr{V}_{\rm SW}^{[n_\alpha]}))$$
  defined by $\sigma\mapsto \mathscr{V}_{\rm SW}^{[n_\alpha]}\otimes_{\HH(S_{r}^{\aff})} \sigma$. Here 
  $$S_r^\aff = Y \rtimes S_r = \Z^r \rtimes S_r$$
  is the extended affine Weyl group of $\GL_r$ and $\HH(S_r^\aff)$ the Iwahori--Hecke algebra of $\GL_r$. We recall the module structure of $\mathscr{V}_{\rm SW}^{[n_\alpha]}$ in \S \ref{VSWright} and \S \ref{SS:Qact}. Moreover, since $\msc{V}_{\rm SW}^{[n_\alpha]}$ has a left action by $U_q(\hat{\mfr{sl}}(n_\alpha))$ which commutes with the right action by $\HH(S_r^\aff)$, one has a natural algebra homomorphism
  $$\Psi: U_q(\hat{\mfr{sl}}(n_\alpha)) \longrightarrow \End_{\HH(S_r^\aff)}(\msc{V}_{\rm SW}^{[n_\alpha]})$$
  which induces a functor
  $$\Psi^*: \mca{M}_L(\End_{\HH(S_{r}^{\aff})}(\mathscr{V}_{\rm SW}^{[n_\alpha]})) \longrightarrow \mca{M}_L(U_q(\hat{\mfr{sl}}(n_\alpha))).$$
 By abuse of language, we also call 
 \begin{equation} \label{QS-Psi}
 \mca{F}_{\rm SW}^\Psi:= \Psi^* \circ \mca{F}_{\rm SW}: \mca{M}_L(\HH(S_r^\aff))  \to \mca{M}_L(U_q(\hat{\mfr{sl}}(n_\alpha)))
 \end{equation}
 the quantum affine Schur--Weyl functor. In the last part of this section, we obtain a relation between the local scattering matrices and the $R$-matrices, utilizing $\Psi^*$ and the relation between $\mca{F}_{\rm GG}$ and $\mca{F}_{\rm SW}$.

  To compare $\mca{F}_{\rm GG}$ and $\mca{F}_{\rm SW}$ we identify $\widehat{\HH}_{I}$ and $\HH(S_r^\aff)$. Let $\wt{G}^\vee$ be the dual group of $\wt{G}$, and let $G_{Q,n}$ be the $F$-points of a split linear algebraic group whose Langlands dual group is $\wt{G}^\vee$. Let $I_{Q,n} \subset G_{Q,n}$ be the Iwahori subgroup, analogous to $I \subset G$. Depending on the choice of a certain Weyl-invariant distinguished $\iota$-genuine character 
$$\theta: Z(\wt{T}) \to \C^\times,$$
it was shown by Savin \cite{Sav04} (see also \cite{GG}) that one has a natural algebra isomorphism
\begin{equation} \label{phi-t}
\phi_\theta: \HH(G_{Q,n}, I_{Q,n}) \to \widehat{\HH}_I.
\end{equation}
Since $\overline{G}$ is a type (\hyperref[C1]{C1}) cover of $\GL_r$, 
 the group $G_{Q,n}$ has root datum
 $$(n_\alpha^{-1} X, \set{n_\alpha^{-1} \alpha}; n_\alpha Y, \set{n_\alpha \alpha^\vee}).$$
 In particular, $G_{Q,n} \simeq \GL_r$. This gives a natural isomorphism $\HH(G_{Q,n}, I_{Q,n}) \simeq \HH(S_r^\aff)$. By abuse of notation, we still write
   \begin{equation} \label{E:idH}
\phi_\theta: \HH(S_r^\aff) \rightarrow \widehat{\HH}_I
 \end{equation}
 for the resulting isomorphism, which by restriction is the identity map on $\HH_W$. One has the induced functor
 $$\phi_\theta^*: \mca{M}_\Diamond(\widehat{\HH}_I) \longrightarrow \mca{M}_\Diamond(\HH(S_r^\aff))$$
 for $\Diamond \in \set{L, R}$.

\subsection{Right $\HH(S_r^\aff)$-action on $\msc{V}_{\rm SW}^{[m]}$ and $\mca{F}_{\rm SW}$} \label{VSWright}
In this subsection, let $m\in \N_{\gest 1}$. Let $\C[\Z]$ be the $\C$-vector space with a basis given by $\set{v_i: i\in \Z}$. Let
$$\msc{V}_{\rm SW}:= (\C[\Z])^{\otimes r}.$$
For any $m\in \N_{\gest 1}$, the space $\msc{V}_{\rm SW}$ is endowed with  commuting $(U_q(\hat{\mfr{sl}}(m)), \HH(S_r^\aff))$-actions on the left and right hand side of $\msc{V}_{\rm SW}$, see \cite{ChPr96, Ant}, where the two commuting actions both depend on $m$.  We will use 
$$\msc{V}_{\rm SW}^{[m]}$$
to emphasize the dependence of the action on $m$. Now we describe the right action of $\HH(S_r^\aff)$ on $\msc{V}_{\rm SW}^{[m]}$; we describe the left action of $U_q(\hat{\mfr{sl}}(m))$ in \S \ref{SS:Qact}.

There is a canonical map
$$v: Y \into \C[Y], \quad y \mapsto v_y,$$
and similarly
$$\tv: Y/m Y \into \C[Y/m Y] = (\C^m)^{\otimes r}, \quad y \mapsto \tv_y.$$
In terms of the standard $\Z$-basis $\set{e_i: 1\lest i \lest r}$  of $Y$, one has
$$\C[Y] = \C[e_1^{\pm 1}, e_2^{\pm 1}, ..., e_r^{\pm 1}].$$
Clearly,
\begin{equation} \label{E:Rm}
\mfr{R}_m:=\set{\sum_{j=1}^r i_j\cdot e_j: 0\lest i_j \lest m -1}
\end{equation}
is a set of representatives of $Y/mY$. For every integer $i_j \in [0, m-1]$ we write
$$\tv_{i_j}:= \tv_{i_j \cdot e_j}.$$
Thus, a basis of $\C[Y/mY] = (\C^m)^{\otimes r}$ is given by
$$\mfr{B}_m=\set{\tv_{i_1} \otimes \tv_{i_2} \otimes ... \otimes \tv_{i_r}: 0 \lest i_j \lest m -1 \text{ for every $j$}}.$$
One has the $\C$-vector space isomorphism
$$\tau_m:  (\C^m)^{\otimes r} \otimes_\C \C[Y] \longrightarrow \msc{V}_{\rm SW} $$
given by
$$ (\tv_{i_1} \otimes ... \otimes \tv_{i_r}) \otimes (e_1^{k_1}\cdot ... \cdot e_r^{k_r}) \mapsto  v_{i_1 - k_1 m} \otimes ... \otimes v_{i_r - k_r m}.$$

The algebra $\C[W]$ acts on $\C[Y/mY] = (\C^m)^{\otimes r}$ naturally, induced from the permutation action of $W$ on $Y/mY$. This representation can be ``deformed" to get a representation $(\gamma_m, (\C^m)^{\otimes r})$ of $\HH_W$. More precisely, for $\tv_{\underline{i}}:=\tv_{i_1} \otimes ... \otimes  \tv_{i_r} \in (\C^m)^{\otimes r}$, and simple reflection $s_{\alpha_k} = (k, k+1) \in W$, one defines (see \cite[Lemma 2.19]{Ant} or \cite[Lemma 3.1]{SSV21})
\begin{equation}
\tv_{\underline{i}} \cdot T_{s_{\alpha_k}} =
\begin{cases}
\tv_{\underline{i} \cdot s_{\alpha_k}} & \text{ if } i_k < i_{k+1}, \\
q^{-1} \cdot \tv_{\underline{i}} & \text{ if } i_k = i_{k+1}, \\
\tv_{\underline{i} \cdot s_{\alpha_k}} + (q^{-1}- q) \cdot v_{\underline{i}} & \text{ if } i_k > i_{k+1},
\end{cases}
\end{equation}
where $T_{s_{\alpha_k}} \in \HH_W$ is the generator associated with $s_{\alpha_k}$. This gives the desired $\HH_W$-action $\gamma_m$ on $(\C^m)^{\otimes r}$. We note that $(\gamma_m, (\C^m)^{\otimes r})$ can be described in an alternative way as follows. Following notations of \cite{Gre99, DDF12, Ant}, let 
$$\Lambda(m, r)$$
denote the set of all $\lambda:=(\lambda_1, \lambda_2, ..., \lambda_m) \in \Z^{m}_{\gest 0}$ satisfying 
$$\sum_{i=1}\lambda_i = r.$$
For each $\lambda \in \Lambda(m, r)$, consider
$$y_\lambda=(\underbrace{1, 1, ..., 1}_{\lambda_1}, \underbrace{2, 2, ..., 2}_{\lambda_2}, ...., \underbrace{m, m, ..., m}_{\lambda_m}) \in (\Z/m\Z)^r.$$
It is clear that the stabilizer subgroup of $y_\lambda$ in $W=S_r$ is 
$$W_\lambda:=W_{y_\lambda} = S_{\lambda_1} \times ... \times S_{\lambda_m}.$$
Then there is a natural $\C$-linear isomorphism 
\begin{equation} \label{E:gamma}
\bigoplus_{\lambda \in \Lambda(m, r)} \mbm{1} \otimes_{\HH_{W_\lambda}}  \HH_W \longrightarrow (\C^m)^{\otimes r},
\end{equation}
given as in \cite[Page 23, (13)]{Ant} for example, such that $\gamma_m$ is the transported action from the left hand side of this isomorphism.

In view of the Bernstein presentation
$$\HH(S_r^\aff) \simeq \HH_W \otimes_\C  \C[Y],$$
we get that as $\C$-vector spaces,
$$\msc{V}_{\rm SW} \simeq (\C^m)^{\otimes r} \otimes_\C \C[Y] \simeq (\C^m)^{\otimes r} \otimes_{\gamma_m, \HH_W} \HH(S_r^\aff).$$
This gives a right action of $\HH(S_r^\aff)$ on $\msc{V}_{\rm SW}$ depending on $\gamma_m$, which we denote by
$$\msc{V}_{\rm SW}^{[m]} \curvearrowleft \HH(S_r^\aff).$$

\subsection{The Iwahori--Matsumoto involution on $\HH(S_{r}^{\aff})$}
Let $l: Y \rtimes W \to \N$ be a length function.  Consider the Iwahori--Matsumoto involution (\cite[Page 48]{IM65})
$${\rm IM}: \mca{H}(S_r^\aff) \longrightarrow \mca{H}(S_r^\aff)$$
specified by
$${\rm IM}(T_w) = (-1)^{l(w)} q^{l(w)} \cdot T_{w^{-1}}^{-1}$$
for all $w \in Y\rtimes W$. Here, $T_w \in \mca{H}(S_r^\aff)$ is the natural element supported by the $I$-double coset given by $w$. 
We have the induced functor
$${\rm IM}^*: \mca{M}_\Diamond(\mca{H}(S_r^\aff)) \longrightarrow \mca{M}_\Diamond(\mca{H}(S_r^\aff))$$
for $\Diamond \in \set{L, R}$. Let
\begin{equation} \label{V-thet}
\msc{V}_{\rm SW, \theta}^{[m]}:= (\phi_\theta^*)^{-1}(\msc{V}_{\rm SW}^{[m]}) \in \mca{M}_R(\widehat{\HH}_I), \quad \mca{V}_\theta:= \phi_\theta^*(\mca{V}^*) \in \mca{M}_R(\HH(S_r^\aff)).
\end{equation}

\begin{prop} \label{P:IM}
Let $\wt{G}=\wt{\GL}_r^{(n)}$ be a cover of $\GL_r$ of type (\hyperref[C1]{C1}). Then $(\phi_\theta^*)^{-1}$ gives a natural isomorphism
\begin{equation} \label{E:rel}
\bold{F}_{\theta}:\mca{V}^{*} \longrightarrow (\phi^{-1}_{\theta})^{*}\circ {\rm IM}^*(\msc{V}_{\rm SW}^{[n_\alpha]})
\end{equation}
of right $\widehat{\HH}_I$-modules and thus a ring isomorphism $\varphi_{\theta}:\End_{\HH(S_{r}^{\aff})}(\msc{V}_{\rm SW}^{[n_\alpha]})\rightarrow \End_{\widehat{\HH}_I}(\mca{V}^{*})$. Moreover, for every $\pi \in \mca{M}_L(\widehat{\HH}_I)$, the isomorphism \eqref{E:rel} induces a canonical $\C$-isomorphism
\begin{equation}\label{GGSWTensorIso}
\bold{F}_{\theta}\otimes {\rm id}_{\pi}:\mca{V}^{*} \otimes_{\widehat{\HH}_I} \pi \stackrel{}{\longrightarrow} \msc{V}_{\rm SW}^{[n_\alpha]} \otimes_{\HH(S_{r}^{\aff})} \rm{IM}^{*}\circ \phi_{\theta}^{*}(\pi),
\end{equation}
which is $\varphi_{\theta}^{-1}$-equivariant with respect to the left actions of $\End_{\widehat{\HH}_I}(\mca{V}^{*})$ and $\End_{\HH(S_{r}^{\aff})}(\msc{V}_{\rm SW}^{[n_\alpha]})$ respectively.
\end{prop}
\begin{proof}
We first prove \eqref{E:rel}. Since $\msc{X}_{Q,n} = Y/n_\alpha Y$, it is easy to see that the following map is a well-defined bijection.
$$\Lambda(n_\alpha, r) \longrightarrow \set{W\text{-orbits in } \msc{X}_{Q,n}}, \quad \lambda \mapsto W(y_\lambda)$$
This immediately gives that
\begin{equation}\label{GGSWEqn2}
\msc{V}_{\rm SW}^{[n_\alpha]} = \bigoplus_{\mathcal{O}_{y}\subset \msc{X}_{Q,n}} \mbm{1}_{W_{y}} \otimes_{\HH_{W_{y}}} \HH(S_{r}^{\aff}).
\end{equation}

To prove \eqref{E:rel} we use the following lemma.

\begin{lm}\label{IMCharSwap}
Let $\chi$ be a character of the parabolic subalgebra $\mathcal{H}_{W_{y}}$. Then the $\mathcal{H}_{W_{y}}$-module homomorphism 
\begin{equation*}
{\rm IM}^{*}(\chi)\longrightarrow {\rm IM}^*(\HH(S_{r}^{\rm aff})\otimes_{\HH_{W_{y}}}\chi)
\end{equation*}
defined by $1 \mapsto 1\otimes 1$ extends to an isomorphism of left $\HH(S_{r}^{\aff})$-modules 
$$\nu:\HH(S_{r}^{\aff})\otimes_{\mathcal{H}_{W_y}}{\rm IM}^{*}(\chi)\longrightarrow {\rm IM}^*(\HH(S_{r}^{\aff})\otimes_{\mathcal{H}_{W_y}}\chi).$$
\end{lm}

\begin{proof}
The linear map ${\rm IM}^{*}(\chi)\rightarrow {\rm IM}^*(\HH(S_{r}^{\aff})\otimes_{\mathcal{H}_{W_{y}}}\chi)$ defined by $1\mapsto 1\otimes 1$ is an $\mathcal{H}_{W_{y}}$-module homomorphism.

By Frobenius reciprocity, this $\mathcal{H}_{W_y}$-module map induces an $\HH(S_{r}^{\aff})$-module homomorphism 
$$\nu:\HH(S_{r}^{\aff})\otimes_{\mathcal{H}_{W_y}}{\rm IM}^{*}(\chi)\longrightarrow {\rm IM}^*(\HH(S_{r}^{\aff})\otimes_{\mathcal{H}_{W_y}}\chi).$$ Note that $\nu(h\otimes 1)={\rm IM}(h)\otimes 1$. We claim that $\nu$ is an isomorphism.

It suffices to construct an inverse map of $\nu$. We consider the linear map 
$$\nu^{\prime}:{\rm IM}^*(\HH(S_{r}^{\aff})\otimes_{\mathcal{H}_{W_y}}\chi)\rightarrow \HH(S_{r}^{\aff})\otimes_{\mathcal{H}_{W_y}}{\rm IM}^{*}(\chi)$$
 defined by $h\otimes 1\mapsto {\rm IM}(h)\otimes 1$. This is well-defined because 
 $$\nu^{\prime}( hh_{y}\otimes 1)= {\rm IM}(h)\otimes \chi(h_{y})=\nu^{\prime}(h\otimes \chi(h_{y})).$$
 Note that as  vector spaces ${\rm IM}^*(\HH(S_r^\aff)\otimes_{\HH_{W_y}}\chi)=\HH(S_r^\aff)\otimes_{\HH_{W_y}}\chi$. Thus we have constructed a $\C$-linear map $\nu^{\prime}: {\rm IM}^*(\HH(S_r^\aff)\otimes_{\HH_{W_y}}\chi)\rightarrow \HH(S_{r}^{\aff})\otimes_{\HH_{W_y}}{\rm IM}^{*}(\chi)$ such that $h\otimes 1\mapsto {\rm IM}(h)\otimes 1$. 

Since $\nu\circ\nu^{-1}(h\otimes 1)=h\otimes1$ and $\nu^{-1}\circ\nu(h\otimes1)=h\otimes 1$, the $\C$-linear map $\nu^{\prime}$ is the inverse of $\nu$. Furthermore, since $\nu$ is an $\HH(S_{r}^{\aff})$-module homomorphism, we see $\nu^{\prime}$ is also an $\HH(S_{r}^{\aff})$-homomorphism. Thus $\nu$ is an $\HH(S_{r}^{\aff})$-module isomorphism.
\end{proof}

Now we can prove \eqref{E:rel}. First, if $\mbm{1}$ is the trivial character of a standard parabolic subgroup $\HH_{W_y} \subseteq \mathcal{H}_{W}$, then ${\rm IM}^{*}(\mbm{1})= \varepsilon$, the sign character of $\HH_{W_y}$. Second, for any two ${\rm IM}^*(\HH(S_{r}^{\aff}))$-modules $V_{1}$ and $V_{2}$ we have 
$$\rm{IM}^*(V_{1}\oplus V_{2})\simeq \rm{IM}^*(V_{1})\oplus \rm{IM}^*(V_{2}).$$
Using Lemma \ref{L:V-mirr}, Lemma \ref{IMCharSwap} and these two facts it follows that there is an isomorphism $\bold{F}_{\theta}:\mca{V}^{*} \longrightarrow (\phi^{-1}_{\theta})^{*}\circ {\rm IM}^*(\msc{V}_{\rm SW}^{[n_\alpha]})$. This completes the proof of equation \eqref{E:rel}. The remaining claims follow directly from it.
\end{proof}

\subsection{Comparison of $\mca{F}_{\rm GG}$ and $\mca{F}_{\rm SW}$} \label{SS:comp}

Recall the isomorphism $\phi_\theta: \HH(S_r^\aff)\to \widehat{\HH}_I$ which, coupled with \eqref{E:rel}, gives rise to a natural algebra isomorphism
$\varphi_\theta: \End_{\HH(S_r^\aff)}(\msc{V}_{\rm SW}^{[n_\alpha]})\to \End_{\widehat{\HH}_I}(\mca{V}^*)$. It induces a functor
$$\varphi_\theta^*: \mca{M}_L(\End_{\widehat{\HH}_I}(\mca{V}^*)) \longrightarrow \mca{M}_L(\End_{\HH(S_r^\aff)}(\msc{V}_{\rm SW}^{[n_\alpha]})).$$
Now we can compare $\mca{F}_{\rm GG}$ and $\mca{F}_{\rm SW}$ as follows.

\begin{thm} \label{T:SW=Wh}
Let $\wt{G}=\wt{\GL}_r^{(n)}$ be a cover of $\GL_r$ of type (\hyperref[C1]{C1}).  Then for every $\pi\in \mca{M}_L(\widehat{\HH}_I)$, one has
\begin{equation} \label{E:SW=Wh}
\mca{F}_{\rm SW} \circ {\rm IM}^* \circ \phi_\theta^*(\pi) \simeq \varphi_\theta^* \circ \mca{F}_{\rm GG}(\pi),
\end{equation}
i.e., the diagram
$$\begin{tikzcd}
\mca{M}_L(\widehat{\HH}_I) \ar[d, "{\rm{IM}^{*}\circ\phi_\theta^*}"]  \ar[r, "\mathcal{F}_{\rm{GG}}"] & \mca{M}_L(\End_{\widehat{\mca{H}}_I}(\mca{V}^*)) \ar[d, "{\varphi_\theta^*}"] \\
\mca{M}_L(\HH(S_r^\aff)) \ar[r, "{\mca{F}_{\rm SW}}"] &  \mca{M}_L(\End_{\HH(S_r^\aff)}(\msc{V}_{\rm SW}^{[n_\alpha]}))
\end{tikzcd}$$
of functors commutes.
\end{thm}
\begin{proof} Using \eqref{E:rel}, the left hand side of \eqref{E:SW=Wh} is isomorphic to
$$\msc{V}_{\rm SW}^{[n_\alpha]} \otimes_{\HH(S_r^\aff)} ({\rm IM}^* \circ \phi_\theta^*(\pi) ) \simeq {\rm IM}^*(\msc{V}_{\rm SW}^{[n_\alpha]}) \otimes_{\HH(S_r^\aff)} \phi_\theta^*(\pi) \simeq  \phi_\theta^*(\mca{V}^*) \otimes_{\HH(S_r^\aff)} \phi_\theta^*(\pi).$$
Since $\mca{F}_{\rm GG}(\pi) = \mca{V}^{*} \otimes_{\widehat{\HH}_I} \pi$
and thus 
$$\varphi_\theta^* \circ \mca{F}_{\rm GG}(\pi) \simeq \phi_\theta^*(\mca{V}^{*}) \otimes_{\HH(S_r^\aff)} \phi_\theta^*(\pi),$$
\eqref{E:SW=Wh} follows.
\end{proof}

The next corollary follows immediately from Theorem \ref{T:SW=Wh}. 

\begin{cor}
Let $\pi$ be a left $\wt{\GL}_r^{(n)}$-representation that is irreducible $\iota^{-1}$-genuine and Iwahori-spherical. Then
$$\dim \Wh_{\psi^{-1}} (\pi) = \dim \Hom_{\widehat{\HH}_I}(\mca{V}, (\pi^{\vee})^I) =\dim \mca{F}_{\rm SW}({\rm IM}^* \circ \phi_\theta^*((\pi^{\vee})^I)),$$
where $\mca{F}_{\rm SW}$ is associated with $\msc{V}_{\rm SW}^{[n_\alpha]}$. Moreover, if $n_\alpha > r$, then $\mca{F}_{\rm SW}$ is a faithful functor (see \cite[Theorem 2.60]{Ant}) and this gives that $\dim \Wh_\psi(\pi)>0$ for every $\pi$ in this case. 
\end{cor}

An analogue of the above equality and the faithfulness for $n_\alpha > r$ were first observed by V. Buciumas. In fact, he already noticed that one can relate Whittaker models of such $\pi$ to quantum affine representations by using the work of Chari--Pressley \cite{ChPr96}.

\subsection{Left $U_q(\hat{\mfr{sl}}(m))$-action on $\msc{V}_{\rm SW}^{[m]}$} \label{SS:Qact}
Let $m\in \N_{\gest 1}$ be general. We recall the quantum affine group $U_q(\hat{\mfr{sl}}(m))$ and its left action on $\msc{V}_{\rm SW}^{[m]}$. First, let $\alpha^\dag$ be the highest root of the root system of type $A$. Write
$$\alpha_0 = - \alpha^\dag.$$
Let $$\hat{\mca{C}}_A:=\left[\angb{\alpha_i}{\alpha_j^\vee}\right]_{0\lest i, j \lest m-1}$$
be the generalized Cartan matrix of type $A_r$. For convenience, we write
$$a_{i, j}=\angb{\alpha_i}{\alpha_j^\vee}.$$
Now, $U_q(\hat{\mfr{sl}}(m))$ is the quantized affine Lie algebra associated with $\hat{\mca{C}}_A$. More precisely, it is the $\C$-algebra with generators
$$\set{E_i, F_i, K_i^{\pm1}: 0 \lest i \lest m-1}$$
and relations (see \cite[\S 2.4]{ChPr96} and \cite[\S 2.5]{Ant})
\begin{enumerate}
\item[(R1)] $K_i K_i^{-1} = 1 = K_i^{-1} K_i$ and $K_i K_j = K_j K_i$.
\item[(R2)]  $K_i E_j = q^{a_{i,j}} \cdot E_j K_i$ and $K_i F_j = q^{-a_{i,j}} \cdot F_j K_i$.
\item[(R3)]  $[E_i, F_j] =\delta_{i, j}\cdot (K_i - K_i^{-1})/(q-q^{-1})$.
\item[(R4)] 
$$\begin{aligned}
E_i^2 E_{i\pm 1} + E_{i\pm 1} E_i^2 & = (q+q^{-1}) E_i E_{i\pm 1} E_i \\
F_i^2 F_{i\pm 1} + F_{i\pm 1} F_i^2 & = (q+q^{-1}) F_i F_{i\pm 1} E_i \\
E_i E_j  & = E_j E_i \text{ if } i-j \ne 0, 1, m-1\\
F_i F_j & = F_j F_i \text{ if } i-j \ne 0, 1, m-1.
\end{aligned}$$
The computation of $i\pm 1$ and $i-j$ above are all valued in $\set{0, 1, ..., m-1}$, the fixed set of representatives of $\Z/m\Z$.
\end{enumerate}

In terms of the basis $\set{v_i: i\in \Z}$ of $\C[\Z]$, one has an action of $U_q(\hat{\mfr{sl}}(m))$ on $\C[\Z]$ given by the following:
$$\begin{aligned}
E_i\cdot v_j & = \delta_{i+1, j} \cdot v_{j-1} \\
F_i \cdot v_j & = \delta_{i, j} \cdot v_{j+1} \\
K_i \cdot v_j & = q^{\delta_{i, j} - \delta_{i+1, j}} \cdot v_j.
\end{aligned}$$
Again, the computation of $i+1$ and $j\pm 1$ in the subscripts of the $\delta$-functions above are valued in the set $\set{0, 1, ..., m-1}$ of representatives for $\Z/m\Z$. 

To have an action of $U_q(\hat{\mfr{sl}}(m))$ on $\msc{V}_{\rm SW}^{[m]}=(\C[\Z])^{\otimes r}$, we note that the quantum group $U_q(\hat{\mfr{sl}}(m))$ has a comultiplication map 
$$\begin{tikzcd}
\Delta: U_q(\hat{\mfr{sl}}(m)) \ar[r] & U_q(\hat{\mfr{sl}}(m))\otimes_\C U_q(\hat{\mfr{sl}}(m))
\end{tikzcd}$$
given by 
$$E_i \mapsto E_i \otimes K_i^{-1} + 1\otimes E_i, \quad F_i \mapsto F_i\otimes 1 + K_i \otimes F_i, \text{ and } K_i \mapsto K_i \otimes K_i.$$
This gives a well-defined representation 
$$\Delta^{r-1}: U_q(\hat{\mfr{sl}}(m))  \longrightarrow \End_\C( (\C[\Z])^{\otimes r} ),$$
equivalently, a left action $U_q(\hat{\mfr{sl}}(m)) \curvearrowright \msc{V}_{\rm SW}^{[m]}$.

It is known (see \cite[Corollary 2.55]{Ant}) that the above left action by $U_q(\hat{\mfr{sl}}(m))$ commutes with the right $\HH(S_r^\aff)$-action given in \S \ref{VSWright} , which we write as
\begin{equation} \label{D:ca}
U_q(\hat{\mfr{sl}}(m)) \curvearrowright \msc{V}_{\rm SW}^{[m]} \curvearrowleft \HH(S_r^\aff).
\end{equation}
From \eqref{D:ca} one has an algebra homomorphism
$$\Psi: U_q(\hat{\mfr{sl}}(m)) \longrightarrow \End_{\HH(S_r^\aff)}(\msc{V}_{\rm SW}^{[m]}),$$
which is surjective if $m > r$, see \cite[Theorem 3.2.1]{DoGr07} or \cite[Theorem 3.8.3]{DDF12}. Here $\Psi$ induces
$$\Psi^*: \mca{M}_L({\rm End}_{\HH(S_r^\aff)}(\msc{V}_{\rm SW}^{[m]})) \longrightarrow \mca{M}_L(U_q(\hat{\mfr{sl}}(m))).$$

\begin{dfn} \label{D:eva}
For every $s\in \C$, the subspace
$$S_{m, s}:={\rm Span}_\C \set{v_i - q^s v_{i+m}: i\in \Z} \subset \C[\Z]$$
is $U_q(\hat{\mfr{sl}}(m))$-stable, and we call $V(s):=\C[\Z]/S_{m, s}$ the shifted evaluation representation of $U_q(\hat{\mfr{sl}}(m))$ associated with $s$.
\end{dfn}
Here $V(s)$ is exactly the $U_q(\hat{\mfr{sl}}(m))$-module $V(q^s)$ in \cite[\S 2.4]{ChPr96}, for $a=q^s$ there. Recall the standard evaluation representation $V_{\rm ev}(b):= (V_q^\natural)_{q^b}, b \in \C$, where the latter is in the notation of \cite[Example 12.3.17]{ChPr95}. Then it follows from loc. cit. that $V(s) = V_{\rm ev}(2/(m+1) + s)$, whence called the shifted evaluation representation.

\begin{rmk} \label{R:GL}
We may consider $U_q(\hat{\mfr{gl}}(m))$ instead of $U_q(\hat{\mfr{sl}}(m))$. In this case, one can define the natural left $U_q(\hat{\mfr{gl}}(m))$ action on $\msc{V}_{\rm SW}^{[m]}$ which still commutes with the right $\HH(S_r^\aff)$-action, depending on $m$. It gives rise to the composite of functors
$$\begin{tikzcd}
\mca{M}_L(\HH(S_r^\aff)) \ar[r, "{\mca{F}_{\rm SW}'}"]  & \mca{M}_L(\End_{\HH(S_r^\aff)}(\msc{V}_{\rm SW}^{[m]})) \ar[r, "{(\Psi')^*}"] &  \mca{M}_L(U_q(\hat{\mfr{gl}}(m))),
\end{tikzcd}$$
which is faithful if $m \gest r$.
This faithfulness can be obtained as follows. First, $\Psi': U_q(\hat{\mfr{gl}}(m)) \onto \End_{\HH_I}(\msc{V}_{\rm SW}^{[m]})$ is surjective if $m \gest 2$, see \cite[Theorem 3.8.1]{DDF12}. On the other hand, it is clear that there exists $\wt{\GL}_r^{(n)}$ of type (\hyperref[C1]{C1}) such that $n_\alpha =m$, for example one can take any $(\mbf{p}, \mbf{q}, n) = (\mbf{p}, 2\mbf{p}m, 2\mbf{p}m)$ with $\mbf{p}\ne 0$. Thus,  if $m \gest r$, 
then it follows from Proposition \ref{P:Meq} and Example \ref{E:GLsta} that $\mca{F}_{\rm Wh}$ gives an equivalence of categories. Since the analogue of \eqref{E:SW=Wh} still holds, we see that $\mca{F}_{\rm SW}'$ is faithful for $n_\alpha \gest r$, and so is $(\Psi')^* \circ \mca{F}_{\rm SW}'$.
\end{rmk}

\subsection{Local scattering matrices and R-matrices} \label{SS:lsm-rm}
By using $\Psi^*$ and the relation between $\mca{F}_{\rm GG}$ and $\mca{F}_{\rm SW}$ in Theorem \ref{T:SW=Wh}, we relate the local scattering matrices with $R$-matrices.

We consider exclusively $m=n_\alpha$ for the remainder of this paper. We write
$$\mca{F}_{\rm GG, \theta}^\Psi:=\Psi^* \circ \varphi_\theta^* \circ \mca{F}_{\rm GG}: \mca{M}_L(\widehat{\HH}_I) \longrightarrow \mca{M}_L(U_q(\hat{\mfr{sl}}(n_\alpha)))$$
and
$$\mca{F}_{\rm SW, \theta}^\Psi:=\mca{F}_{\rm SW}^\Psi \circ \phi_\theta^*: \mca{M}_L(\widehat{\HH}_I) \longrightarrow \mca{M}_L(U_q(\hat{\mfr{sl}}(n_\alpha)))$$
to simplify notation, where $\mca{F}_{\rm SW}^\Psi$ is as in \eqref{QS-Psi}.
By Theorem \ref{T:SW=Wh}, for every $\pi \in \mca{M}_L(\widehat{\HH}_I)$  we have
\begin{equation}\label{PsiGGSW}
\mca{F}_{\rm GG, \theta}^\Psi(\pi) \simeq \mca{F}_{\rm SW}^\Psi({\rm IM}^{*} \circ \phi_\theta^*(\pi)).
\end{equation}

First, we consider the principal series representation of $\wt{\GL}_r^{(n)}$ of type (\hyperref[C1]{C1}). We have 
$$\wt{T} \simeq (\wt{T}_1 \times ... \times \wt{T}_r)/H,$$
where $H=\set{(\zeta_j)_j: \prod_j \zeta_j=1}$ and $\wt{T}_j \subset \wt{T}$ is the covering torus associated to $Y_j:= \Z e_j$. One has $Z(\wt{T}_j) = \wt{(T_j)^{n_\alpha}}$, and we set
$$\theta_j:=\theta|_{Z(\wt{T}_j)},$$
which is an unramified $\iota$-genuine character. Consider the character 
$$\chi_{s_j}: T_j \simeq F^\times \longrightarrow \C^\times$$
 given by $\chi_{s_j}(a^{e_j}) = \val{a}^{s_j}$. This gives
$$i_\theta(\chi_{s_j}):= i(\theta_j) \otimes \chi_{s_j} \in \Irr_\iota(\wt{T}_j),$$
where $i(\theta_j) \in \Irr_\iota(\wt{T})$ is the unique $\iota$-genuine representation with central character $\theta_j$
determined by the Stone-von Neumann theorem. For $s=(s_1, ..., s_r) \in \C^r$, we have 
$$i_\theta(\chi_s): = \otimes_j i_\theta(\chi_{s_j}) \in \Irr_\iota(\wt{T})$$
and thus the $\iota$-genuine principal series representation $\tilde{\mbm{I}}_\theta(\chi_s):={\rm Ind}_{\wt{B}}^{\wt{G}}(i_\theta(\chi_s))$, where $\wt{G}$ acts on the left by right translation. Naturally, $\tilde{\mbm{I}}_\theta(\chi_s)^I$ is a left $\widehat{\HH}_I$-module. 
Similarly, one has the unramified principal series $\mbm{I}(\chi_s)$ of $\GL_r$ with left action, and thus $\mbm{I}(\chi_s)^I \in \mca{M}_L(\HH(S_r^\aff))$. We have
\begin{equation} \label{pt-ps}
\phi_\theta^*(\tilde{\mbm{I}}_\theta(\chi_s)^I) = \mbm{I}(\chi_{n_\alpha s})^I.
\end{equation}


Since 
$$\mca{F}_{\rm SW}^\Psi: \mca{M}_L(\HH(S_r^\aff)) \longrightarrow \mca{M}_L(U_q(\hat{\mfr{sl}}(n_\alpha)))$$
is a monoidal functor (see \cite[Proposition 4.8]{ChPr96}), one has
\begin{equation} \label{E:mon}
\mca{F}_{\rm SW, \theta}^\Psi(\tilde{\mbm{I}}_\theta(\chi_s)^I) \simeq  \mca{F}_{\rm SW, \theta}^\Psi(i_\theta(\chi_{s_1})^{I_1}) \otimes ... \otimes \mca{F}_{\rm SW, \theta}^\Psi(i_\theta(\chi_{s_r})^{I_r}),
\end{equation}
where 
$$I_j = I \cap T_j = T_j(O_F).$$
In fact, we can directly verify \eqref{E:mon} as follows.

\begin{prop} \label{P:FSW-V}
For every $1\lest j \lest r$ one has 
\begin{equation} \label{E:SW-V}
\mca{F}_{\rm SW, \theta}^\Psi(i_\theta(\chi_{s_j})^{I_j}) \simeq V(n_\alpha s_j)
\end{equation}
as left $U_q(\hat{\mfr{sl}}(n_\alpha))$-modules, where $V(n_\alpha s_j)$ is the shifted evaluation representation in Definition \ref{D:eva}.
Moreover,
\begin{equation} \label{E:SW-ten}
\mca{F}_{\rm SW, \theta}^\Psi(\tilde{\mbm{I}}_\theta(\chi_s)^I) \simeq V(n_\alpha s_1) \otimes ... \otimes V(n_\alpha s_j) \otimes ... \otimes V(n_\alpha s_r).
\end{equation}
\end{prop}
\begin{proof}
Fix $j$ and write $Y_j=\Z e_j, Y_{Q,n, j}= \Z n_\alpha e_j$ and $\msc{X}_{Q,n, j} = Y_j/Y_{Q,n, j}$. Consider $\HH(S_1^\aff) = \C[Y_j]$. We see that $\HH(S_1^\aff) \simeq \C[Y_{Q,n, j}] = \C[\Z n_\alpha e_j]$ acts naturally on $\msc{V}_{\rm SW}^{[n_\alpha]} \simeq \C[\msc{X}_{Q,n, j}]  \otimes_\C \C[Y_{Q,n,j}]$. Now,
$$\mca{F}_{\rm SW, \theta}^\Psi(i_\theta(\chi_{s_j})^{I_j}) =\msc{V}_{\rm SW}^{[n_\alpha]}  \otimes_{\HH(S_1^\aff)}  \phi_\theta^* \left( i_\theta(\chi_{s_j})^{I_j}\right)  \simeq \msc{V}_{\rm SW}^{[n_\alpha]} \otimes_{\HH(S_1^\aff)} (\chi_{n_\alpha s_j})^{I_j},$$
where the second isomorphism follows from \eqref{pt-ps}. It is now easy to check the isomorphism in \eqref{E:SW-V}.

For the second isomorphism, we note $\phi_\theta^*(\tilde{\mbm{I}}_\theta(\chi_{s})^I) \simeq \mbm{I}(\chi_{n_\alpha s})^I \simeq  \HH(S_r^\aff) \otimes_{\C[Y]} (\chi_{n_\alpha s})^{I \cap T} $ and thus
$$
\begin{aligned}
\mca{F}_{\rm SW, \theta}^\Psi(\tilde{\mbm{I}}(\chi_s)^I) & \simeq \msc{V}_{\rm SW}^{[n_\alpha]}   \otimes_{\HH(S_r^\aff)}   \mbm{I}(\chi_{n_\alpha s})^I \\
& \simeq (\msc{V}_{\rm SW}^{[n_\alpha]}|_{\C[Y]})  \otimes_{\C[Y]} (\chi_{n_\alpha s})^{I \cap T}\\
&  \simeq \bigotimes_{j=1}^r ( \C[\msc{X}_{Q,n,j}] \otimes \C[Y_j])  \otimes_{\C[Y_j]} (\chi_{n_\alpha s_j})^{T_j} \\
& \simeq \bigotimes_{j=1}^r \mca{F}_{\rm SW, \theta}^\Psi(i_\theta(\chi_{s_j})^{I_j}).
\end{aligned}
$$
This gives \eqref{E:mon}. Coupled with \eqref{E:SW-V}, we obtain \eqref{E:SW-ten}.
\end{proof}

\begin{cor}
The unramified principal series $\tilde{\mbm{I}}_\theta(\chi_s)$ is reducible if and only if $q = q^{n_\alpha (s_i - s_j)}$ for some $i\ne j$.
\end{cor}
\begin{proof}
This follows immediately from combining Proposition \ref{P:FSW-V} and \cite[\S 4.8 Corollary]{ChPr96}. In fact, more directly, since $\widehat{\mca{H}}_I \simeq \mca{H}(S_r^{\rm aff})$ for $\wt{\GL}_r$ of type (\hyperref[C1]{C1}), the points of reducibility of $\tilde{\mbm{I}}_\theta(\chi_s)$ are exactly those of its Shimura lifted principal series $\mbm{I}(\chi_{n_\alpha s})$ of $\GL_r$ modulo an $n_\alpha$-scaling, the latter is well-known as given in the statement.
\end{proof}

In view of the Theorem \ref{T:SW=Wh}, the isomorphism in \eqref{E:mon} is an incarnation of Rodier's heredity of Whittaker models \cite{Rod1}. We make this precise as follows. First, one has a natural $\C$-isomorphism
$$\eta(\chi_{s}): \mca{F}_{\rm SW, \theta}^\Psi(\tilde{\mbm{I}}_\theta(\chi_{s})^I) \longrightarrow \mca{F}_{\rm GG, \theta}^\Psi(\tilde{\mbm{I}}_\theta(\chi_{s})^I),$$
by forgetting the additional structure in line (\ref{PsiGGSW}).

For every $j$ one has the natural $\C$-isomorphisms
\begin{equation} \label{E:L1}
 \C[\Z e_j/\Z n_\alpha e_j] \simeq \mca{F}_{\rm GG, \theta}^\Psi( i_\theta(\chi_{s_j})^{I_j}) \simeq \mca{F}_{\rm SW, \theta}^\Psi(i_\theta(\chi_{s_j})^{I_j}) \simeq V(n_\alpha s_j).
\end{equation}
Also, the same argument as in the proof of Proposition \ref{P:FSW-V} gives the $\C$-isomorphisms
\begin{equation} \label{E:Wh-ten}
\begin{aligned}
\mca{F}_{\rm{GG}, \theta}^\Psi(\tilde{\mbm{I}}_\theta(\chi_{s})^I) &  \simeq  \mca{V}_\theta \otimes_{\HH(S_r^\aff)} \mbm{I}(\chi_{n_\alpha s})^I \\
& \simeq \bigotimes_{j=1}^r ( \C[\msc{X}_{Q,n,j}] \otimes \HH(S_1^\aff))  \otimes_{\HH(S_1^\aff)} (\chi_{n_\alpha s_j})^{T_j}  \\
& \simeq \bigotimes_{j=1}^r \mca{F}_{\rm GG, \theta}^\Psi( i_\theta(\chi_{s_j})^{I_j})  \simeq \bigotimes_{j=1}^r V(n_\alpha s_j).
\end{aligned}
\end{equation}
It is easy to see that the diagram
\begin{equation} \label{CD:SWWH}
\begin{tikzcd}
 & \bigotimes_{j=1}^r V(n_\alpha s_j) \\
\mca{F}_{\rm SW, \theta}^\Psi(\tilde{\mbm{I}}_\theta(\chi_s)^I) \ar[rr, "{\eta(\chi_s)}"] \ar[ru] & &  \mca{F}_{\rm GG, \theta}^\Psi(\tilde{\mbm{I}}_\theta(\chi_{s})^I)   \ar[lu]
\end{tikzcd}
\end{equation}
of $\C$-vector spaces commutes. Here the left and right slanted arrows are given by \eqref{E:SW-ten} and \eqref{E:Wh-ten} respectively.

We can interpret equation (\ref{E:Wh-ten}) as Rodier heredity in the following way. Let $\tilde{\mbm{I}}_{\theta^{-1}}(\chi_{-s})$ be the $\iota^{-1}$-genuine principal series where $\overline{G}$ acts on the right by right translation $(f\cdot g)(x):= f(x g^{-1})$ for $g\in \wt{G}$ and $f\in \tilde{\mbm{I}}_{\theta^{-1}}(\chi_{-s})$. (Note that this arises from the usual left action of $\overline{G}$ by right translation via the antiinvolution $g\mapsto g^{-1}$.) Then one has an isomorphism 
\begin{equation} \label{E:vee-iso}
\tilde{\mbm{I}}_{\theta^{-1}}(\chi_{-s})^{\vee}\simeq \tilde{\mbm{I}}_\theta(\chi_{s})
\end{equation}
 of $\iota$-genuine left $\wt{G}$-representations. Taking $I$-fixed vectors gives isomorphisms
\begin{equation*} 
(\tilde{\mbm{I}}_{\theta^{-1}}(\chi_{-s})^I)^\vee \simeq (\tilde{\mbm{I}}_{\theta^{-1}}(\chi_{-s})^{\vee})^I \simeq (\tilde{\mbm{I}}_\theta(\chi_{s}))^I
\end{equation*}
 of left $\widehat{\HH}_{I}$-modules. For $\pi$ a right $\widehat{\HH}_I$-module we have $\mca{F}_{\rm{Wh},\theta}^{\Psi}(\pi)=\mca{F}_{\rm{GG},\theta}^{\Psi}(\pi^{\vee})$. It then follows from \eqref{E:Wh-ten} that
\begin{equation*}
\mca{F}_{\rm{Wh}, \theta}^\Psi(\tilde{\mbm{I}}_{\theta^{-1}}(\chi_{-s})^I)\simeq \bigotimes_{j=1}^r \mca{F}_{\rm Wh, \theta}^\Psi( i_{\theta^{-1}}(\chi_{-s_j})^{I_j}).
\end{equation*}

Now we turn to intertwining operators. For every simple root $\beta = \beta_k$ with $1\lest k \lest r-1$, we have the simple reflection $w_\beta =(k, k+1) \in S_r$ associated with $\beta$. Consider the intertwining map 
$$T(w_\beta): \tilde{\mbm{I}}_\theta(\chi_s) \longrightarrow \tilde{\mbm{I}}_\theta({}^{w_\beta} \chi_s)$$
 as in McNamara \cite[Equation (7.2)]{Mc1}, where they are written $T_{w_{\beta}}$. Note that since $\theta$ is Weyl-invariant, the above notations in the definition of $T(w_\beta)$ are indeed well-defined.  Furthermore, with McNamara's normalization these intertwining operators satisfy the braid relations.
 
 Passing to $I$-fixed vectors gives a map of left $\widehat{\mca{H}}_{I}$-modules $\tilde{\mbm{I}}_\theta(\chi_s)^{I} \longrightarrow \tilde{\mbm{I}}_\theta({}^{w_\beta} \chi_s)^{I}$, which we still write as $T(w_{\beta})$. 
Note ${}^{w_\beta}\chi_s = \chi_{w_\beta(s)}$ and further $w_\beta$ acts on $s=(s_1, ..., s_j, ..., s_r)$ via that on the index $j$. Consider the intertwining map
$$T^\sharp(w_\beta): \tilde{\mbm{I}}_{\theta^{-1}}({}^{w_\beta}\chi_{-s})^I \longrightarrow \tilde{\mbm{I}}_{\theta^{-1}}(\chi_{-s})^I.$$
induced from $T(w_\beta)$ and \eqref{E:vee-iso}. In view of $\mca{F}_{\rm{Wh},\theta}^{\Psi}(\pi)=\mca{F}_{\rm{GG},\theta}^{\Psi}(\pi^{\vee})$, we have the following commutative diagram arising from \eqref{CD:SWWH}:


\begin{equation} \label{CD:1}
\begin{tikzcd}
& \bigotimes_{j=1}^r V(n_\alpha s_j) \arrow[from=dl] \arrow[from=rd] \arrow{dd}[near start]{\mfr{F}(w_\beta)^*} &  \\ 
\mca{F}_{\rm SW, \theta}^\Psi(\tilde{\mbm{I}}_\theta(\chi_s)^I) \arrow[rr, crossing over, "{\qquad \qquad \qquad \eta(\chi_s)}"'] \arrow[dd, "{\mca{F}_{\rm SW, \theta}^\Psi(T(w_\beta))}"'] & &   \mca{F}_{\rm Wh, \theta}^\Psi(\tilde{\mbm{I}}_{\theta^{-1}}(\chi_{-s})^I)\\
& \bigotimes_{j=1}^r V(n_\alpha s_{w_\beta(j)}) \arrow[from=dl] \arrow[from=rd] &  \\
\mca{F}_{\rm SW, \theta}^\Psi(\tilde{\mbm{I}}_\theta({}^{w_\beta}\chi_s)^I) \arrow[rr, "{\eta({}^{w_\beta}\chi_s)}"'] &  & \mca{F}_{\rm Wh, \theta}^\Psi(\tilde{\mbm{I}}_{\theta^{-1}}({}^{w_\beta}\chi_{-s})^I) \arrow[from=uu, "{ \mca{F}_{\rm Wh, \theta}^\Psi(T^\sharp(w_\beta)) }"] .
\end{tikzcd}
\end{equation}

Here, the front square diagram, and in particular the two maps $\mca{F}_{\rm SW, \theta}^\Psi(T(w_\beta))$ and $\mca{F}_{\rm Wh, \theta}^\Psi(T^\sharp(w_\beta))$ are induced from $T(w_\beta)$ and $T^\sharp(w_\beta)$. The map $\mfr{F}(w_\beta)^*$ naturally arises from the front square in view of \eqref{CD:SWWH}.
In fact, by the functorial and monoidal property of $\mca{F}_{\rm SW}$, one has (noting that $w_\beta = (k, k+1) \in S_r$)
$$\mfr{F}(w_\beta)^* = {\rm id} \otimes  ... \otimes \mfr{F}_k(w_\beta) \otimes ...  \otimes {\rm id}$$
where 
$$\mfr{F}_k(w_\beta):   V(n_\alpha s_k) \otimes V(n_\alpha s_{k+1})  \longrightarrow V(n_\alpha s_{k+1}) \otimes V(n_\alpha s_k) $$
is a  $U_q(\hat{\mfr{sl}}(n_\alpha))$-homomorphism. 

It is clear from \eqref{CD:1} that $\mfr{F}(w_\beta)^*$ can be represented by any matrix form of $\mca{F}_{\rm Wh, \theta}^\Psi(T^\sharp(w_\beta))$. We write
$$s^*= -w_\beta(s) = (-s_1, -s_2, ..., -s_{k+1}, -s_k, ..., -s_r) \in \C^r.$$
The operator $\mca{F}_{\rm Wh, \theta}^\Psi(T^\sharp(w_\beta))$ is represented by a so-called local scattering square matrix 
$$[\mca{S}(y', y; w_\beta, \chi_{s^*})]_{y', y\in \mfr{R}_{n_\alpha}},$$
 of size $\val{\msc{X}_{Q,n}}$, where $\mfr{R}_{n_\alpha} \subset Y$ is the set of representatives of $\msc{X}_{Q,n}$ in \eqref{E:Rm}. This matrix  was first studied by Kazhdan--Patterson \cite{KP} for $\wt{\GL}_r$ and then generalized by McNamara \cite{Mc2} for general $\wt{G}$.  More precisely, for every $y= \sum_j y_j e_j\in \mfr{R}_{n_\alpha}$ with $y_j \in [0, n_\alpha -1]$, one has
 $$\lambda^{-s}_y \in \mca{F}_{\Wh, \theta}^\Psi(\tilde{\mbm{I}}_{\theta^{-1}}(\chi_{-s})^I)\simeq \mca{F}_{\rm{GG}, \theta}^\Psi(\tilde{\mbm{I}}_{\theta}(\chi_{s})^I),$$
 the naturally associated element via \eqref{E:L1} and \eqref{E:Wh-ten}. The  matrix $[\mca{S}(y', y; w_\beta, \chi_{s^*})]$ is then by definition the one satisfying
\begin{equation} \label{F:sca}
\mca{F}_{\rm Wh, \theta}^\Psi(T^\sharp(w_\beta))(\lambda_{y'}^{-s}) = \sum_{y \in \mfr{R}_{n_\alpha}}  \mca{S}(y', y; w_\beta, \chi_{s^*}) \cdot \lambda_y^{s^*}.
\end{equation} 

Note that we also have $u_{y_j}^{s_j} \in V(n_\alpha s_j)$ for every $y_j$ as above, using \eqref{E:L1}. Hence, for $y = \sum_j y_j e_j \in \mfr{R}_{n_\alpha}$, one has
  $$u_y^s:= u_{y_1}^{s_1} \otimes ... \otimes u_{y_j}^{s_j} \otimes ... \otimes u_{y_r}^{s_r} \in V(n_\alpha s_1) \otimes ... \otimes V(n_\alpha s_r),$$
which corresponds to  $\lambda_y^{-s}$ in the right top slanted arrow in \eqref{CD:1}. Similar correspondence holds for $\lambda_{y}^{s^*}$ and $u_{y}^{-s^*}$.  The commutativity of \eqref{CD:1} and the discussion above immediately give the following.

 \begin{thm} \label{T:sca}
Let the domain and codomain of the $U_q(\hat{\mfr{sl}}(n_\alpha))$-homomorphism $\mfr{F}(w_\beta)^*$ be endowed with the bases $\set{u_{y'}^s:  y' \in \mfr{R}_{n_\alpha}}$ and $\set{u_y^{-s^*}: y\in \mfr{R}_{n_\alpha}}$ respectively. Then $\mfr{F}(w_\beta)^*$ is represented by the local scattering matrix $[\mca{S}(y', y; w_\beta, \chi_{s^*})]_{y', y \in \mfr{R}_{n_\alpha}}$.
 \end{thm}
 
Consider the matrix 
$$[\mca{M}(y', y; w_\beta, \chi_{s^*})]:=[\mca{S}(w_\beta(y'), y; w_\beta, \chi_{s^*})]_{y, y' \in \mfr{R}_{n_\alpha}}$$
and let  
$$R(w_\beta, \chi_{s^*}) \in \End_\C(V(n_\alpha s_1) \otimes ... \otimes V(n_\alpha s_r))$$
be the $\C$-endomorphism  represented by $[\mca{M}(y', y; w_\beta, \chi_{s^*})]_{y, y' \in \mfr{R}_{n_\alpha}}$ with respect to the basis $\set{u_y^s: y\in \mfr{R}_{n_\alpha}}$. 
It is then easy to see that
$$\mfr{F}(w_\beta)^*  = \tau(w_\beta) \circ R(w_\beta, \chi_{s^*})$$
where $\tau(w_\beta)$ is the obvious $\C$-homomorphism induced from the flipping map
$$ V(n_\alpha s_k) \otimes V(n_\alpha s_{k+1}) \longrightarrow V(n_\alpha s_{k+1}) \otimes V(n_\alpha s_k), \quad v\otimes w \mapsto w \otimes v.$$

As a consequence of \eqref{CD:1} and that the $T(w_{\beta})$ satisfy the braid relations, the $\mfr{F}(w_\beta)^*$'s satisfy the braid relations as well. Thus, the operators $R(w_\beta, \chi_{s^*})$ solve the quantum Yang--Baxter equation, and the representing matrices $[\mca{M}(y', y; w_\beta, \chi_{s^*})]$  can be properly called $R$-matrices. One can check easily that it agrees with the one given in \cite[Page 105]{BBB19}. We also note that the matrix $[\mca{M}(y', y; w_\beta, \chi_{s^*})]$ is exactly a ``local coefficients matrix" associated with $\tilde{\mbm{I}}_{\theta^{-1}}(\chi_{s^*})$ and $w_\beta$ as in  \cite[\S 3.2]{GSS2}.

\section{Some remarks} \label{S:rmk}
The functor $\mca{F}_{\rm GG}$ is defined for any root system type, while
$$\mca{F}_{\rm SW}^\Psi: \mca{M}_L(\HH(S_r^\aff)) \longrightarrow \mca{M}_L(U_q(\hat{\mfr{sl}}(n_\alpha)))$$
is for type $A$ only. It is a natural question to ask for a generalization of $\mca{F}_{\rm SW}^\Psi$, if we replace the domain or codomain of $\mca{F}_{\rm SW}^\Psi$ by more general root system types.

First, for a general semisimple Lie algebra $\mfr{g}$ it was shown by \cite{KKK18,Fuj20} that there exists a natural functor
$$\mca{F}_{\rm SW}^{\mfr{g}}: \bigoplus_{\beta} \mca{M}_L(R^J(\beta)) \longrightarrow \mca{M}_L(U_q(\hat{\mfr{g}})),$$
where $R^J(\beta)$ is certain quiver algebra given by Khovanov--Lauda \cite{KhLa09} and Rouquier \cite{Rou}. Many results are established regarding $\mca{F}_{\rm SW}^{\mfr{g}}$, especially those pertaining to the categorification of the two sides of $\mca{F}_{\rm SW}^{\mfr{g}}$.

On the other hand, instead of $\HH(S_r^\aff)$ in $\mca{F}_{\rm SW}^\Psi$ above, if one considers affine Hecke algebras of general Cartan type, for example of type B or C, then the recent work of W.-Q. Wang,  H.-C. Bao and their collaborators \cite{BKLW18, BWW18, FLLLWW} have made advances in the framework of quantum symmetric pairs. In particular, a  Schur--Weyl type duality between affine Hecke algebras of type $B_r$ or $C_r$ and certain coideal subalgebra of $U_q(\mfr{gl}(k))$ is established. We refer the reader to loc. cit. and references therein for details.


In the geometric setting, the work of Gaitsgory and Lysenko \cite{Gai08, GaLy18, Lys17} represents some significant progress towards understanding a certain Whittaker category in the metaplectic setting and also its relation to the quantum groups in the framework of the quantum Langlands program (see \cite{Gai16d}). It will be very interesting to see what would be the analogue of such results in the classical $p$-adic context, and conversely what role the quantum affine Schur--Weyl functor plays in the formulation of some geometric results.



\end{document}